\documentclass[12pt,A4paper]{article}
\usepackage{blindtext}
\usepackage[utf8]{inputenc}
\usepackage{amsmath,amssymb,amsthm}
\usepackage{bm,dcolumn}

\usepackage{indentfirst}

\usepackage{graphicx}
\usepackage{hyperref}
\newtheorem{thm}{Theorem}[section]
\newtheorem{lem}[thm]{Lemma}
\newtheorem{prop}[thm]{Proposition}

\def\one{\mbox{1\hspace{-4.25pt}\fontsize{12}{14.4}\selectfont\textrm{1}}} 
\usepackage{algorithm,algorithmic}
\usepackage{color}
\usepackage{colortbl}
\usepackage{multirow}
\usepackage{subfig}
\usepackage{ulem}

\usepackage{geometry}
\usepackage[pagewise]{lineno}

\title{\Large\textbf{Dewetting dynamics of anisotropic particles \\
-- a level set numerical approach}}
\author{Siddharth Gavhale  and Karel Svadlenka\\ \small Department of Mathematics, Kyoto University, Japan}
{}
\date{}

\begin{document}
\maketitle

\noindent
{\bf Abstract.} We extend thresholding methods for numerical realization of mean curvature flow on obstacles to the anisotropic setting where interfacial energy depends on the orientation of the interface. This type of schemes treats the interface implicitly, which supports natural implementation of topology changes, such as merging and splitting, and makes the approach attractive for applications in material science. The main tool in the new scheme are convolution kernels developed in previous studies, that approximate the given anisotropy in a nonlocal way. We provide a detailed report on the numerical properties of the proposed algorithm. \\

\vspace{0.05cm}


\section{Introduction}
Evolution of solid or liquid particles possessing direction-dependent surface energies on substrate is a phenomenon appearing in several fields of applied science and engineering, such as cell biology or material science \cite{Bao2017,Wang2015, Winterbottom1967,Xu2017,Jiang2012,Campinho2013}.
For example, in coating techniques such as thermal spraying, it is important to predict the dynamics of spreading of impinging particles \cite{Bertagnolli1995}.
Likewise, in manufacturing of nanopatterned substrates, the precise control of the size and location of forming nanoparticles is essential to boost the functionality of the product \cite{Huang2005}.
The above applications concern two opposing types of dynamics: either several small particles merge into a thin film or a thin films splits into particles forming a pattern. It is an ongoing challenge to become able to control these dynamics by tuning physical parameters such as thickness of the film or surface energy of the material \cite{Thompson2012}.

Mathematical models usually express the motion of particle as an interface evolution problem via gradient flow of an energy functional with respect to some metric. When an obstacle is present in the form of substrate, one arrives at a multiphase problem, whence the evolution is also governed by force balance at triple points, i.e., points where three different phases meet. 
For instance, \cite{Wang2015} formulates a surface diffusion problem on obstacle with volume constraint and numerically implements it using explicit front-tracking method.
A large number of numerical studies on  explicit treatment of interface tracking, such as finite differences, marker particle and finite elements, has appeared in recent years and this approach is now well developed. 
It turns out to be very effective when changes of topology are absent from the evolution.
However, when topology changes occur, such as merging and splitting of particles, explicit methods require a manual surgery, which is not based on any mathematical principle and may be prohibitively complicated, especially in higher dimensions. 

It is known that implicit representation of the interface, e.g., as a level set of a function, allows for graceful handling of topology changes. Nevertheless, there are only a few studies on its numerical implementation for motion of particles on substrates. 
We mention here two prominent results that we are aware of: the phase field method for solving a Cahn-Hilliard model introduced by Jiang et al. in \cite{Jiang2012}, and the thresholding approach for mean curvature flow developed by Xu et al. in \cite{Xu2017}.  
Although both of these works confirmed good behavior of the proposed schemes, they address only the isotropic case, where the energy does not depend on orientation of interface.

The purpose of this article is to extend the thresholding method of \cite{Xu2017} to the anisotropic setting, which is motivated by the numerical advantages of the thresholding method: unconditional stability, low computational cost and natural handling of topology changes. In this generalization, we apply the theory of anisotropic convolution kernels developed in a series of papers \cite{Bonnetier2012, Elsey2018, Esedoglu2017b}. Since the numerical properties of these kernels were not yet systematically studied, we first provide a numerical analysis of the two-phase problem, i.e., without obstacle. The main contribution of the paper is the construction and numerical investigation of an algorithm for solving the full problem of a volume-preserving anisotropic particle evolving on substrate, and possibly undergoing topology changes.
We show that although the new scheme retains the good points of thresholding algorithms, each of the studied kernels has certain drawbacks, pointing to new research directions.

\section{The model}
\label{sec_model}

First, let us describe the setup of the problem.
We consider a particle $P$ on a rigid substrate $S$ surrounded by a vapor region $V$, where $P,S,V$ are taken as closed sets (see Figure \ref{fig:setup}).
Since we are interested solely in the evolution of the particle, it does not make any difference if we frame our system in a bounded domain $\Omega \subset \mathbb{R}^d$ that is large enough so that the particle does not touch its boundary during the evolution. 
We deal mainly with the case $d=2$ and only briefly comment on the case of general $d$. 
For later use we denote $\Omega^{\text{up}} := P \cup V=\overline{\Omega\setminus S}$ and also introduce the symbols $\Gamma_{SP}= S \cap P, \Gamma_{SV} = S \cap V$ and $\Gamma = \Gamma_{PV} = P \cap V$ for the interface between substrate--particle, substrate--vapor and particle--vapor, respectively.
Their surface energies will be denoted by $\gamma_{SP}, \gamma_{SV}$ and $\gamma_{PV}$.	
{{We assume $\gamma_{SP}$ and $\gamma_{SV}$ are constant along the respective interface but the particle--vapor interface has orientation-dependent  energy, i.e., $\gamma_{PV}(x) = \gamma(\boldsymbol{n}(x))$, where $\boldsymbol{n}$ is the outer normal  to $P$ at a point  $ x \in \Gamma_{PV}$.
Since the normal $\boldsymbol{n}$ is uniquely identified as $(\cos \theta, \sin \theta)$, where $\theta \in [-\pi, \pi)$ is the angle between  $\boldsymbol{n}$ and the positive direction of $y$-axis measured clockwise from the $y$-axis,
the function $\gamma$ can be considered as a function of one variable $\theta$,
namely, $\gamma(\boldsymbol{n}(x)) =  \gamma(\cos\theta(x),\sin\theta(x)) =: \widetilde{\gamma}(\theta(x))$. In the subsequent text we will omit the tilde for simplicity.}}
The substrate $\Gamma_{S} = \Gamma_{SP} \cup \Gamma_{SV}$ is fixed throughout the evolution but the contact (or free boundary) points $x_{c}^{l}$ and $x_{c}^r$ of the particle with substrate may move due to the deformation of $\Gamma$. 
The area $A$ of the particle region $P$ is assumed to be preserved during the evolution. 
	\begin{figure}[!ht]
		\centering
		\includegraphics[width=0.55\textwidth]{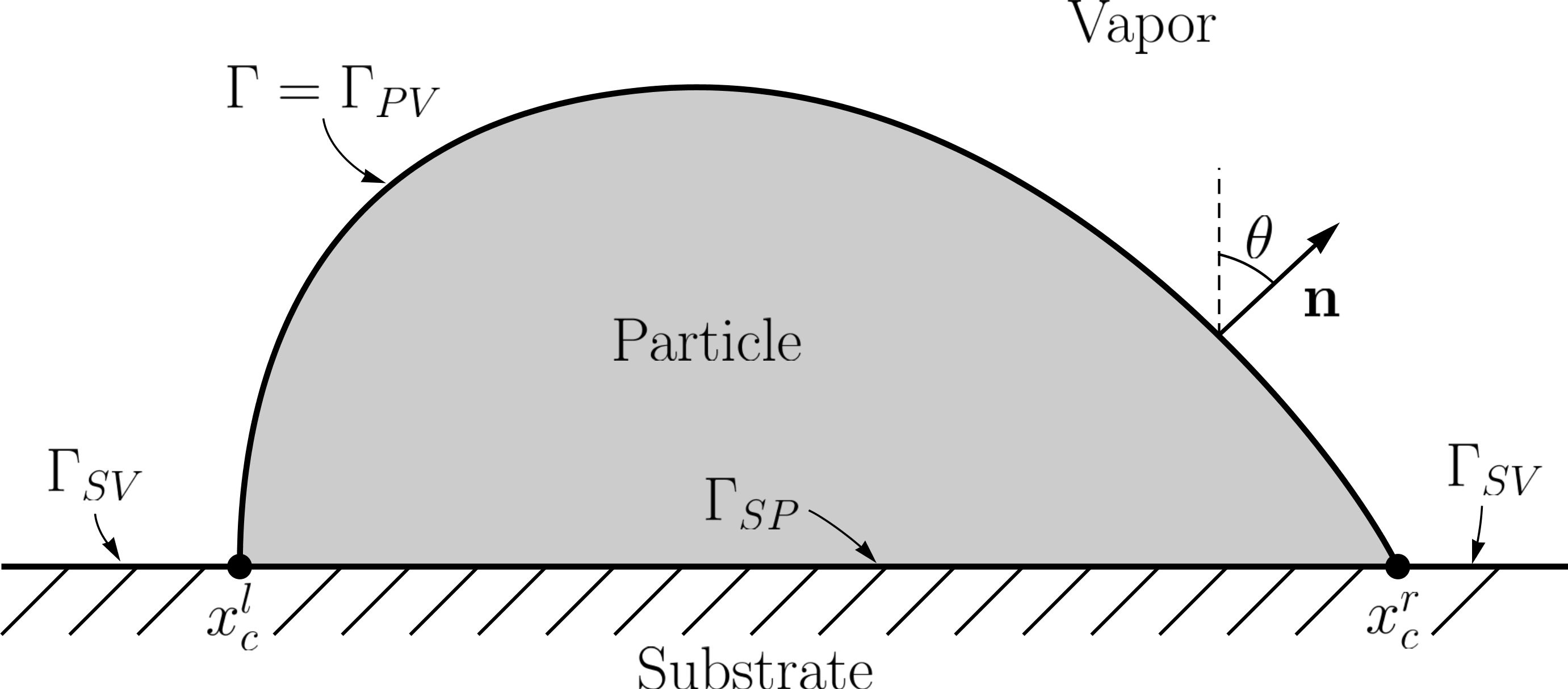}
        		\caption{Setup and notation: a particle on a flat, rigid substrate.} 
		\label{fig:setup}
	\end{figure}
	
Total interfacial energy of this system is given by 
    \begin{equation}
    \label{eq:energy}
 	E(\Gamma) = \int_{\Gamma} \gamma_{PV} \, ds + \int_{\Gamma_{SP}} \gamma_{SP} \, ds  + \int_{\Gamma_{SV}} \gamma_{SV} \, ds ,
	\end{equation}
and equilibrium shapes of particle $P$ are the (local) minima of this energy under the area constraint $|P|=A$. 
The authors in \cite{Bao2017} define equilibrium shapes for $d=2$ as those curves $\Gamma$ at which the first variation with respect to area-preserving normal perturbations and arbitrary tangential perturbations of $\Gamma$ vanishes, and show that this is equivalent to the following two conditions: 
    \begin{eqnarray}
    && {{\kappa (x) \left( \gamma''(\theta(x)) + \gamma(\theta(x)) \right) = C \qquad \text{a.e.} \; x,}}
    \label{eq:firstvar} \\
    && \gamma (\theta) \cos \theta - \gamma'(\theta) \sin \theta + \gamma_{SP} - \gamma_{SV} = 0 \qquad \theta = \theta_c^l, \theta_c^r.
    \label{eq:aniyoung}
    \end{eqnarray}
{{Here $\kappa (x)$ is the curvature of $\Gamma$ at a given point $x$}}, $C$ is a constant determined from the area of particle, and $\theta^l_c, \theta^r_c$ are the angles at the left and right contact points, respectively.
For isotropic surface energy ($\gamma =$ constant), \eqref{eq:aniyoung} reduces to the well-known Young's equation, and thus we call \eqref{eq:aniyoung} the {\it anisotropic Young's equation}, cf. \cite{Wang2015}.
Moreover, based on the second variation, \cite{Bao2017} defines {\it stable equilibria} and shows that a necessary condition for stability is {{
    \begin{equation}
    \label{eq:stab}
        \gamma''(\theta(x)) + \gamma(\theta(x)) \geq 0 \qquad \text{a.e.} \; x.
    \end{equation}
    }}
When $\gamma$ satisfies this condition for all $\theta$, equilibrium shapes can be obtained using the Winterbottom construction, which essentially truncates the Wulff shape corresponding to the anisotropy $\gamma$ at a suitable height determined by the participating surface energies $\gamma_{SP}$ and $\gamma_{SV}$, see \cite{Winterbottom1967}.
The term {\it Wulff shape} indicates the equilibrium shape with least surface energy weighted by the anisotropy $\gamma$ under prescribed area.
For a given anisotropy $\gamma$ it is given by {{
	\begin{equation}
	\label{eq:wulffshape}
	    \mathcal{W}_\gamma 
	    = \left\{ \eta\in{\mathbb{R}}^d \, | \;\;  \eta \cdot \xi \leq 1 \quad \forall \xi \in B_\gamma \right\},
	\end{equation}
where $B_\gamma$ is the unit ball of $\gamma$, named as {\it Frank diagram}, i.e., $B_\gamma = \{ \xi\in{\mathbb{R}}^d\, | \; \gamma(\xi) \leq 1 \}$. Here $\gamma$ is assumed to be extended in a 1-homogeneous way to the whole $\mathbb{R}^d\setminus \{ 0\}$ by $\gamma(\xi)= |\xi| \gamma(\frac{\xi}{|\xi|})$.  }} 
For $d=2$ and anisotropies satisfying \eqref{eq:stab}, the boundary of $\mathcal{W}_\gamma$ can be parametrized as \cite{Wang2015}
    \begin{equation}
    \label{eq:Wulff}
        x(\theta) = -\gamma(\theta) \sin \theta - \gamma'(\theta) \cos \theta, \qquad y(\theta) = \gamma (\theta) \cos \theta - \gamma'(\theta) \sin\theta .
    \end{equation}

The quantity $\gamma'' + \gamma$ turns out to play an important role in determining the properties of the anisotropy function $\gamma$. In particular, we distinguish the following types of anisotropies according to the sign of this quantity (see Figure \ref{fig:Wulff}):
    \begin{enumerate}
    \item[(1)] {\it isotropic} : $\gamma =$ positive constant,
	\item[(2)] {\it weakly anisotropic}: $\gamma (\theta) + \gamma''(\theta) > 0$ for all $\theta$,
    \item[(3)] {\it strongly anisotropic}: there is $\theta$ such that $\gamma(\theta) + \gamma''(\theta) < 0$,
	\end{enumerate}
{\it Wulff envelope} is the curve defined by equations \eqref{eq:Wulff} but for strong anisotropies it does not coincide with the boundary of Wulff shape because it forms "ears" due to self-intersection.
The above classification pertains only to smooth anisotropies $\gamma$ but in applications it often happens that the Wulff shape is a polytope, in which case we call the anisotropy {\it crystalline}. 
In this paper, we deal only with weak anisotropies $\gamma$ with at least $C^2$ smoothness, and therefore we assume that crystalline anisotropies are appropriately regularized.
		
	\begin{figure}[!ht]
	    \centering
	    \includegraphics[width=\textwidth]{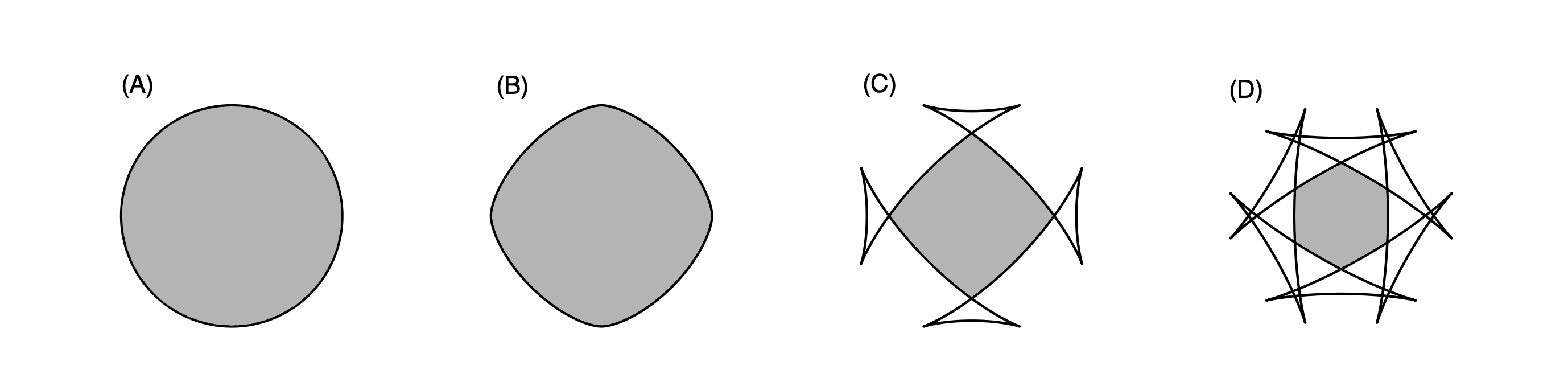}
	    \caption{Wulff envelopes (black line) and Wulff shapes (gray region) for anisotropies of the form $\gamma (\theta)  = 1+ \beta \cos (m\theta)$: (A) isotropic ($\beta=0$), (B) weakly anisotropic ($\beta=0.05, m=4$), (C) strongly anisotropic ($\beta=0.25, m=4$), and (D) strongly anisotropic with intersecting ears ($\beta=0.25, m=6$).}
	\label{fig:Wulff}
	\end{figure}

The aim of this article is to numerically analyze evolution of particles towards equilibrium in two dimensions, including possible topology changes.
In order to do that, we select one of the simplest evolution laws, namely the $L^{2}$-gradient flow of the energy \eqref{eq:energy}, which is usually called the {\it weighted mean curvature flow}. 
In the simple case of one particle undergoing no topology change, there are only two contact points, and a standard derivation analogous to \eqref{eq:firstvar}--\eqref{eq:aniyoung} yields the following evolution problem: 
    \begin{eqnarray}
&& {{ V_\perp (x) = - \mu(\theta(x)) \kappa (x) \left( \gamma''(\theta(x)) + \gamma(\theta(x)) \right) + C, }} 
    \label{eq:evol1}\\
        && \gamma (\theta) \cos \theta - \gamma'(\theta) \sin \theta + \gamma_{SP} - \gamma_{SV} = 0 \qquad \theta = \theta_c^l, \theta_c^r,
    \label{eq:evol2}
    \end{eqnarray}
where $V_\perp$ is the outward normal velocity of the $\Gamma$-interface, the positive function $\mu$ is the mobility of this interface and $C$ is a constant.
Notice that the sign of $\gamma''+\gamma$ is again important here, since when $\gamma''(\theta)+\gamma(\theta)$ is negative for some angles $\theta$, equation \eqref{eq:evol1} becomes backwards parabolic and thus ill-posed.
	
When topology changes occur, such as merging and splitting of particles, equation \eqref{eq:evol1} still holds for smooth parts of $\Gamma$ away from singularities but it is not anymore possible to describe the evolution fully using simple formulas such as \eqref{eq:evol1}--\eqref{eq:evol2}.
For smooth weak anisotropies, one can give a precise mathematical definition in terms of functions of bounded variation in a similar manner to Definition 1.1 of multiphase mean curvature flow in \cite{Laux2016}.
Here we develop a numerical scheme that automatically deals with topology changes but since a rigorous convergence proof for our scheme is out of the scope of this paper, we omit theoretical details. 

\section{Thresholding and anisotropic kernels}
\label{sec:twophase}

In this section, we present a numerical method for computing weighted mean curvature flows. The method is based on the level set approach since our focus is on natural treatment of topology changes.

\subsection{Thresholding scheme}
\label{sec:BMO}

To explain the basic ideas, we consider a free particle $P$ in a vapor domain $\Omega \subset {{\mathbb{R}^d}}$, without any substrate.
Denoting the particle--vapor interface by $\Gamma$ and it's orientation-dependent energy density by  {{$\gamma (\boldsymbol{n})$}}
, the total energy of this system is simply 
\begin{equation}
\label{eq:energy2}
E(\Gamma) = \int_{\Gamma} \gamma {{(\boldsymbol{n}) \hspace{.1cm}ds   }} .
\end{equation}
It is known that if the initial shape of the particle satisfies certain conditions, the weighted mean curvature flow of this energy shrinks the particle to a point while asymptotically approaching the Wulff shape corresponding to the anisotropy $\gamma$.

One of effective level set based methods to numerically realize such evolution is the {\it BMO algorithm} proposed in \cite{BMO1994}. 
It was originally designed for isotropic energies and repeats two steps: convolution with Gaussian kernel and thresholding.
Later it was discovered that replacing the Gaussian with a suitable kernel $K$ yields weighted mean curvature flows, see Algorithm 1.
\begin{algorithm}

	\caption{Anisotropic two-phase BMO algorithm} 
		Given a time step $\delta t$ and a particle region $P^k \subset \mathbb{R}^d$ at time $t_k$, to get new region $P^{k+1}$ at next time step $t_{k+1} = t_k+ \delta t$, perform the following two steps:
			\begin{eqnarray}
			\label{BMO2phase1}
				&&\text{Convolution:} \qquad \; U^{k}   = K_{\delta t} \ast \one_{P^{k}}\\
			\label{BMO2phase2}
				&&\text{Thresholding:} \qquad P^{k+1}  =  \left\{ x \,| \;\; U^{k} (x) \geq \frac{1}{2} \int_{\mathbb{R}^d} K {{(y) \, dy }} \right\}
			\end{eqnarray}
	Here $\one_{P^{k}}$ is the characteristic function of the set $P^{k}$ and $K_{\delta t} {{(x)}} = \frac{1}{(\delta t)^{d/2}} K \left( \frac{x}{\sqrt{\delta t}} \right)$.
\end{algorithm} 

Let us briefly comment on the convergence and stability of this algorithm.
It was proved in \cite{Ishii1999} that if the kernel $K$ is {\it positive}, satisfies
\begin{equation}
\label{eq:kernelassum}
K(x) \in L^1 (\mathbb{R}^d), \; xK(x) \in L^1 (\mathbb{R}^d), \;  K(x) = K(-x), \; \int_{\mathbb{R}^d} K(x) \, dx = 1, 
\end{equation}
and {{several other technical conditions (see (3.2)-(3.4) and (3.7) in \cite{Ishii1999} for details)}}, then the above algorithm converges as $\delta t\to 0$ to the viscosity solution of weighted mean curvature flow given by the equation 
$$ u_t(t,x) = \left( \int_{Du(x)^{\perp}} K(y) \, d{{\mathcal{H}^{d-1}}}(y) \right)^{-1} \left(\frac{1}{2} \int_{Du(x)^{\perp}} \langle D^2u(x)y,y \rangle K(y) \, d{{\mathcal{H}^{d-1}}}(y) \right),$$
{{where $Du$ and $D^2u$ denote the gradient vector and Hessian matrix of $u$,
$\mathcal{H}^k$ denotes the $k$-dimensional Hausdorff measure and  $\eta^\perp$ is the orthogonal complement of a vector $\eta$.}}
The positivity of the kernel is essential since it guarantees a comparison principle, on which the theory of viscosity solutions is based.
On the other hand, if we ignore topology changes and are interested only in the normal velocity of a smooth interface, a formal proof of convergence based on Taylor expansions is straightforward (see, e.g., the Appendix of \cite{Elsey2018}).

The unconditional gradient stability of Algorithm 1 was proved in \cite{Esedoglu2015} in the following form:
\begin{prop}
	\label{prop:stability}
	 Let $K$ satisfy \eqref{eq:kernelassum}.
	If $\widehat{K}\geq 0 $, where $\widehat{K}$ is the Fourier transform of $K$, then for any time step size $\delta t>0$, Algorithm 1 decreases the energy 
	\begin{equation}
\label{eq:lyapunov}
 E_{{\delta t}}  (P, K) = \frac{1}{\sqrt{\delta t}}\int_{\mathbb{R}^d \setminus P}  K_{{\delta t}} \ast \one_{P} \, dx
    \end{equation}
	at every time step.
\end{prop}
The energy $E_{\delta t}$ is an approximation to the anisotropic perimeter \eqref{eq:energy2} in the sense that if \eqref{eq:kernelassum} holds and $P$ is a compact subset of $\mathbb{R}^d$ with smooth boundary, then by \cite{Elsey2018},
\begin{equation}
    \label{eq:limlya}
    \lim_{\delta t\to 0} E_{\delta t} (P,K) = \int_{\partial P} \gamma_K(n(x)) \, d{{\mathcal{H}^{d-1}}}(x), \quad \text{where} \;\; \gamma_K (n) = \frac{1}{2} \int_{\mathbb{R}^d} {{| n \cdot y| K({y}) \, dy}} .
\end{equation}
Here, $n(x)$ denotes the outer unit normal to $P$ at a point $x\in\partial P$. The proof of Proposition \ref{prop:stability} in \cite{Elsey2018, Esedoglu2015} clarifies the role of the positive Fourier transform $\widehat{K}$: Algorithm 1 is equivalent to the minimization of the linearization of a relaxed version of the energy $E_{\delta t}$ and the positive Fourier transform guarantees the concavity of this relaxed energy.
The paper \cite{Elsey2018} also derives the form of the corresponding mobility:
\begin{equation}
    \label{eq:kermob}
    \mu_K (n) = \left( \int_{n^{\perp}} K(x) \, d{{\mathcal{H}^{d-1}}}(x) \right)^{-1} .
\end{equation}
This means that one step of Algorithm 1 moves a smooth interface with the normal speed {{
$$ V_{\perp}(x) = - \mu_K(n(x)) \kappa (x) \left( \gamma_K''(n(x)) + \gamma_K(n(x)) \right) $$
}}
to leading order in $\delta t$, where $\gamma_K,\mu_K$ are given by \eqref{eq:limlya} and \eqref{eq:kermob}, respectively.
{{This means that over a short time $\delta t$ the position of a point on the interface obtained by the algorithm deviates from the position prescribed by the above equation by a distance of order at most $\delta t^2$.}} 

To summarize, we have both convergence and stability of Algorithm 1 if $K\geq 0$ and $\widehat{K} \geq 0$ hold at the same time, in addition to further regularity properties of the kernel. 
However, it turns out to be impossible to satisfy these conditions simultaneously for certain anisotropies, as the following theorem shows.

\begin{thm}[\cite{Elsey2018}]
	\label{thm_barrier}
	Threshold dynamics algorithm  \eqref{BMO2phase1} and \eqref{BMO2phase2}
	with a positive kernel can approximate a given weighted mean curvature flow 
	if and only if the Wulff shape corresponding to the anisotropy is a zonoid. 
	Moreover, if the Wulff shape is not a zonoid then a positive convolution kernel cannot be found for any other anisotropy the Wulff shape of which is close enough in the Hausdorff metric. 
\end{thm}
In 2D, every centrally symmetric body is a zonoid but in 3D the condition becomes much more restrictive -- for example, the regular octahedron is not a zonoid.
This problem is addressed in \cite{Esedoglu2017}, where a refined version of Algorithm 1 is developed, which involves more convolution steps but allows for a much wider class of anisotropies and mobilities.
On the other hand, it is important to note that stability and convergence of Algorithm 1 hold, under certain conditions, in arbitrary spatial dimension, and its practical implementation in higher dimensions does not entail any technical complications, which is a major advantage over explicit methods.

\subsection{Anisotropic kernels}
\label{sec_aniker}

When a kernel is given, the properties of the corresponding thresholding scheme can be studied as explained in the previous section. 
The inverse problem of finding a suitable convolution kernel for a given anisotropy and mobility has been investigated over the last two decades.  
The first partially successful attempt was made in \cite{Ruuth2000} where kernels in the form of characteristic function of a carefully chosen set were constructed for two-phase anisotropic evolution in two dimensions.
However, as pointed out later in \cite{Elsey2018}, these kernels do not correctly separate the effect of anisotropic energy and mobility, and thus are not applicable to multi-phase evolution such as the one considered in this paper. 
This can be observed from equations \eqref{eq:evol1}-\eqref{eq:evol2}, where the effect of surface energy and mobility are inseparably combined in the normal velocity but the contact angle condition depends solely on the surface energy.
For this reason, we will review here only kernels developed after this seminal work.

\textbf{Bonnetier--Bretin--Chambolle kernels} \cite{Bonnetier2012} (abbreviated as BBC below) were developed by linearizing in Fourier space the nonlinear diffusion equation corresponding to the evolution by a given anisotropy. 
This led the authors to {{the following expression of the kernel in Fourier domain:}}
$$  \widehat{K}(\xi)=  e^{- 4\pi^2 \gamma^2(\xi)}, $$
{{Here $\gamma$ is the 1-homogeneous extension of the anisotropy function and the Fourier transform is defined by}}
$$ \widehat{K} (\xi) = \int_{{\mathbb R}^d} K(x) e^{-2 \pi i x \cdot \xi} \, dx .$$
An advantage of this type of kernels is that they are always positive in Fourier domain and thus by Proposition \ref{prop:stability} guarantee the stability of the thresholding scheme.
On the other hand, they are not always positive in physical domain and thus lack a proof of convergence. 
For general anisotropies these kernels show a slow decay, which is inefficient from the viewpoint of numerical calculations.
Moreover, the mobility, being fixed by the construction to the natural mobility  $\mu = \gamma$, cannot be freely prescribed.

\textbf{Elsey--Esedoglu kernels} (EE) \cite{Elsey2018} are constructed so that they are positive in both physical and Fourier domain provided that the Wulff shape of the anisotropy is a zonoid. 
The idea is to write the $d$-dimensional kernel as a weighted sum of smoothed one-dimensional Gaussians defined in a direction $\nu \in \mathbb{S}^{d-1} {{ = \{ \xi \in \mathbb{R}^d \, | \, \| \xi \| = 1\} }}$ by
$$ g_{\nu, \epsilon} (x) =\frac{1}{(4\pi)^{\frac{d}{2}}} \hspace*{.1cm} \exp \left(-\frac{(x \cdot \nu)^2}{4}\right) \hspace*{.1cm} \frac{1}{\epsilon^{d-1}} \exp \left(\frac{(x \cdot \nu)^{2} - ||x||^2}{4\epsilon^2}\right),$$
(here the second exponential is a smoothing of $\delta$-function using a parameter $\epsilon$) and solve \eqref{eq:limlya} to obtain the correct weight. 
This yields the formula
\begin{equation}
\label{eq:EEkernel}
K_{\epsilon} (x) = \sqrt{\pi} \int_{\mathbb{S}^{d-1}} (T^{-1} \gamma) {{(\nu)}} g_{\nu, \epsilon} (x) \, d{{\mathcal{H}}}^{d-1}(\nu),
 \end{equation}
where $T^{-1}$ is the inverse cosine transform, which can be in 2D expressed simply by
\begin{equation}
\label{eq:EEinvcos}
T^{-1} \gamma(\theta)  = \frac{1}{4} {{ \left(\frac{d^2}{d\theta^2} + I\right)}}  \hspace*{.1cm}  \gamma \left( \theta + \frac{\pi}{2}\right).
\end{equation}
{{Here, as above, $\gamma(\theta)$ is a shorthand for $\gamma(\cos\theta,\sin\theta)$.}}
 A drawback of this construction is that the mobility is fixed to 
\begin{equation}
    \label{eq_mobee}
    \mu_{\epsilon}(n) = \sqrt{\pi}\left( \int_{\mathbb{S}^{d-1}} (T^{-1} \gamma) {{(\nu)}} \left( (1-\epsilon^2)(\nu \cdot n)^2 + \epsilon^2 \right)^{-1/2} \, d {{\mathcal{H}}}^{d-1}(\nu) \right)^{-1} ,
\end{equation}
and thus cannot be freely designed (note the difference of this formula from the corresponding one in \cite{Elsey2018}).

\textbf{Esedoglu--Jacobs--Zhang kernels} (EJZ) \cite{Esedoglu2017b} are most general because they can realize essentially any admissible anisotropy--mobility pair. Admissibility here means the smoothness and convexity assumptions on this pair required for the existence of solution. The authors construct kernels in 2D and in 3D of two types: one that is positive in the real domain (for zonoidal anisotropies) and one in Schwartz class that is positive in the Fourier domain.

{{(I) Kernel positive in real domain:}} The construction starts with fixing the polar form of the kernel to $K(r, \theta) = \alpha(\theta ) \eta(r \beta(\theta))$, {{ along with}}
$$ \eta (x)= 
\begin{cases}
\exp\left(\frac{-1}{x^2 (x-2)^2}\right),& \text{if } x\in (0,2)\\
0,              & \text{otherwise}
\end{cases} \;\; $$
and proceeds with determining the unknown functions $\alpha,\beta$ by solving \eqref{eq:limlya} and \eqref{eq:kermob}. This results in
\begin{equation}
    \label{ab1}
 \alpha^2 (\theta) = {\frac{m_2}{32 m_0^3 \mu^2(\theta - \frac{\pi}{2}) \sigma (\theta - \frac{\pi}{2})} }, \quad 
 \beta^2 (\theta) = {\frac{m_2}{2 m_0  \sigma (\theta - \frac{\pi}{2})} },
 \end{equation}
where $\sigma(\theta) = \mu(\theta) (\gamma'' (\theta) + \gamma (\theta))$ and {{$ m_j = \int_{0}^{2} x^j \eta(x)\, dx, \; j=0,2$}}.\\

{{(II) Kernel positive in Fourier domain:}} The ansatz is 
$$ \widehat{K} (\xi)= \frac{1}{2} \exp (-\zeta(\alpha(\xi))) + \frac{1}{2} \exp (-\zeta(\beta(\xi))),$$ 
where $ \zeta : \mathbb{R} \rightarrow \mathbb{R}$ is a positive, smooth and even function satisfying 
$\zeta(x)=0$ if $|x|\leq 1$ and $\zeta(x)=x^2$ if $|x| \geq 2$. This is a modification of BBC kernel with the purpose of eliminating the singularity of the Fourier transform of BBC kernel which may appear at the origin for certain anisotropies, and thus improving decay properties of the kernel that are important for numerical efficiency.
By solving \eqref{eq:limlya} and \eqref{eq:kermob} again, now with mobility multiplied by a constant $c^2$ to ensure solvability, one obtains specific formulas for $\alpha$ and $\beta$ as
\begin{equation}
\label{ab2}
\alpha(n) = \tfrac{\pi}{s_2c} (q(n) + \sqrt{q^2(n)-r(n)}), 
\hspace*{.2cm}\beta(n) = \tfrac{\pi}{s_2c} (q(n) - \sqrt{q^2(n)-r(n)}),
\end{equation}
where $q(n) = c\gamma(n)$, $r(n) = 8s_0 s_2 \mu(n) \gamma(n)$, and $s_0=\frac{1}{4\pi} \int_{{\mathbb{R}}} e^{-\zeta(\xi)}\, d\xi$, $s_2=\frac{1}{4\pi} \int_{{\mathbb{R}}} \frac{1-e^{-\zeta(\xi)}}{\xi^2}\, d\xi$ (see \cite{Esedoglu2017b} for details).
We remark that the formulas \eqref{ab1} and \eqref{ab2} contain errors in the original paper.
 Moreover, our formulas differ also due to different definition of Fourier transform.

\subsection{Numerical performance of kernels}
In order to assess the performance of each of the above kernels in thresholding algorithms, we carried out three types of numerical tests: convergence analysis for a simple anisotropy (with ellipse as Wulff shape), experimental analysis for crystalline anisotropy (with square as Wulff shape) and simulation with a complex nonconvex geometry.

In the numerical tests we use Algorithm 1 
to advance the interface. 
The convolution in the diffusion step of the algorithm is effectively calculated employing the Fast Fourier Transform (FFT) yielding
\begin{equation}
\label{eq:fft}
    \widehat{U}_k = \widehat{K}_{\delta t} \widehat{\one}_{P^k},
\end{equation}
and then taking the inverse transform to get $U_k$.
{{The computational domain was chosen as $[-5,5]\times [-5,5]$ for all numerical tests in this paper. It was discretized into a rectangular grid with grid cells of size $dx \times dx$ by dividing both spatial coordinates into the same number of subintervals of length $dx$.}}
The thresholding step in the algorithm then amounts to performing a check at each grid point whether the value of diffused function $U^k$ is greater or less than $\frac{1}{2}\int K$, and assigning the value $1$ to $P^{k+1}$ at the grid point if the value is greater, and the value $0$ otherwise.
However, it is well known that this approach leads to non-smooth behavior of errors when the spatial and temporal grids are refined. The error decreases on the whole with refinement of grid but shows sudden jumps due to interface pinning on grid points. To avoid this unwanted effect in numerical analysis, we implement a simple idea for obtaining a sub-grid spatial accuracy, presented already in previous works \cite{Esedoglu2017b}. 
Namely, we calculate the intersections of the $\frac{1}{2}\int K$-level set of $U^k$ with grid lines, use this information to obtain the ratio of area that each phase (i.e., in this two-phase case, $P^{k+1}$ and ${\mathbb{R}}^2\setminus P^{k+1}$) occupies in every grid cell, and based on these ratios in 4 grid cells common to a grid point assign a value between $0$ and $1$ to $P^{k+1}$ at that grid point.

\subsubsection{Convergence analysis}
\label{sec:conv2phase}
Since numerical behavior of anisotropic kernels has not yet been investigated in detail elsewhere, we first carried out a convergence analysis in 2D for the simple elliptic anisotropy $\gamma (x,y) = \sqrt{(ax)^2 + (by)^2}$ with $a=2$, $b=1$, and the natural mobility $\mu = \gamma$. This setup fulfills all conditions required for the above theoretical results to hold.
Solving \eqref{eq:wulffshape}, or equivalently \eqref{eq:Wulff}, one finds that the corresponding Wulff shape is an ellipse with boundary $(\frac{x}{b})^2 + (\frac{y}{a})^2 = 1$.

In the numerical tests, we start with initial condition identical to the Wulff shape, and apply Algorithm 1 to evolve it by the weighted mean curvature flow \eqref{eq:evol1} with $\mu = \gamma$ (except for EE kernels) and $C=0$. 
It is known (see, e.g., \cite{Yagisita2006,Sevcovic2011}) that the analytical solution is self-similar: the initial Wulff shape shrinks in size without changing its shape. 
To obtain the speed of shrinkage, write the evolving Wulff envelope \eqref{eq:Wulff} as $W_{\gamma}(\theta,t) = \eta(t) W_{\gamma}(\theta,0)$ and notice that by construction we have $W_{\gamma}(\theta,t)\cdot n(\theta) = -\gamma(\theta)$, where $n$ is the unit normal. Moreover, the anisotropic curvature $(\gamma + \gamma'')\kappa$ is equal to $1$ for the Wulff envelope $W_{\gamma}(\theta,0)$ and scales as $1/\eta(t)$ for $W_{\gamma}(\theta,t)$.
Hence, from the formula for the normal velocity $V_{\perp} = \partial_t W_{\gamma} \cdot n = \partial_t \eta (-\gamma) = \gamma (\gamma + \gamma^{''})\kappa = \gamma / \eta$,
we arrive at the ODE $\partial_t \eta = - 1/ \eta$ for $\eta(t)$ with initial condition $\eta(0)=1$, yielding
\begin{equation}
\label{eq:selfsim}
W_{\gamma}(\theta,t)  = \sqrt{1-2t} \;\; W_{\gamma}(\theta,0), \qquad t \in (-\infty, \tfrac{1}{2}].
\end{equation}
This formula indicates that the ellipse shrinks to a point at $t=\frac{1}{2}$.

Convergence order and efficiency of kernels were investigated by executing Algorithm 1
for various combinations of spatial mesh size $dx$ and time step $\delta t$, and for each of the kernels introduced in Section \ref{sec_aniker}.
The initial condition was chosen as the ellipse $x^2+ (\frac{y}{2})^2 =1$.
\begin{table}
\scriptsize
\noindent
\begin{center}
Bonnetier-Bretin-Chambolle kernel ($dx=0.00061$, $N=2^{14}$)\\
		\begin{tabular}{|l|r|r|r|r|r|r|} 
			\hline 
			Time step $\delta t$ &    0.03125  &   0.01562 & 0.00781&  \cellcolor[gray]{0.8} 0.00390 &   0.00195 &    0.00097 \\ 
			\hline
			$ L^{\infty}$-error &  0.05836 &   0.02564 & 0.01207 & \cellcolor[gray]{0.8}   0.00589 &  0.00297&   0.00161 \\ \hline
			Convergence order     &	-	&	{1.19} & {1.09} &	{\cellcolor[gray]{0.8} 1.04} & {0.99} &	{0.88}   \\ \hline
			CPU time (min) & 14 & 28 & 57 & \cellcolor[gray]{0.8} 113 & 227 & 453 \\ \hline
		\end{tabular}
\vspace{.01cm}\\			
\noindent
Esedoglu-Jacobs-Zhang kernel: positive in Fourier ($dx=0.00061$, $N=2^{14}$)	\\
		\begin{tabular}{|l|r|r|r|r|r|r|} 
			\hline 
			Time step $\delta t$ & \cellcolor[gray]{0.8}  0.00390 &   0.00195 &    0.00097&    0.00048 & 0.00024 & 0.00012 \\ 	\hline
		    $ L^{\infty}$-error &  {\cellcolor[gray]{0.8}  0.64610 } &   {0.54342} &  { 0.35265} &    0.09938 &    0.02964 &    0.01246  \\ \hline
			Convergence order      &\cellcolor[gray]{0.8} 	-	& { 0.25}  & {0.62 }   & {1.82}    & 1.75   &  1.25  \\ \hline
			CPU time (min) & \cellcolor[gray]{0.8}    25 & 47 & 95 & 191 & 382 & 760 \\ \hline
		\end{tabular}
			
\noindent
Esedoglu-Jacobs-Zhang kernel: positive in physical ($dx=0.00122$, $N=2^{13}$)\\
		\begin{tabular}{|l|r|r|r|r|r|r|} 
			\hline 
			Time step $\delta t$ &   0.01562 & 0.00781& \cellcolor[gray]{0.8}  0.00390 &   0.00195 &    0.00097&    0.00048  \\ \hline
		    $ L^{\infty}$-error &      0.08137 &    0.04094 &  \cellcolor[gray]{0.8}  0.02058 &  0.01041 & 0.00540 &   0.00308  \\ \hline
			Convergence order &	-	&	 0.99  &  0.99 \cellcolor[gray]{0.8}    & 0.99   &  0.95  &   0.80 \\ \hline
			CPU time (min) & 7 & 13 & \cellcolor[gray]{0.8} 26 & 53 & 106 & 212  \\ \hline
		\end{tabular}
			
\noindent
Elsey-Esedoglu kernel: $\epsilon = 0.1$ ($dx=0.00122$, $N=2^{13}$)\\
		\begin{tabular}{|l|r|r|r|r|r|r|} 
			\hline 
			Time step $\delta t$ & 0.12500 & 0.06250	&    0.03125  &   0.01562 & 0.00781&  \cellcolor[gray]{0.8} 0.00390   \\  \hline
		    $ L^{\infty}$-error &   0.34198 & 0.20097 &  0.13425 &  0.10275 &    0.08994 &  \cellcolor[gray]{0.8}  0.08240  \\ \hline
			Convergence order &	-	&		0.77  &  0.58  &  0.39   &  0.19  & \cellcolor[gray]{0.8}  0.13   \\ \hline
			CPU time (min) & 7 & 13 & 27 & 53 & 107 & \cellcolor[gray]{0.8} 213 \\ \hline
		\end{tabular}
			
\noindent
Elsey-Esedoglu kernel: $\epsilon = 0.05$ ($dx=0.00122$, $N=2^{13}$)\\
		\begin{tabular}{|l|r|r|r|r|r|r|} 
			\hline 
			Time step $\delta t$ & 0.12500 & 0.06250	&    0.03125  &   0.01562 & 0.00781& \cellcolor[gray]{0.8}  0.00390   \\  \hline
		    $ L^{\infty}$-error &    0.36138 &  0.17987 &  0.09401 & 0.05003 & 0.02953 & \cellcolor[gray]{0.8} 0.02281   \\ \hline
			Convergence order    &	-	&		1.01   & 0.94 &0.91 & 0.76 & \cellcolor[gray]{0.8} 0.37 \\ \hline
			CPU time (min)& 8 & 18 & 36 & 70 & 133 &\cellcolor[gray]{0.8}  265  \\ \hline
		\end{tabular}
\vspace{.08cm}\\	
\noindent
Elsey-Esedoglu kernel: $\epsilon = 0.01$ ($dx=0.00122$, $N=2^{13}$)\\
		\begin{tabular}{|l|r|r|r|r|r|r|} 
			\hline 
			Time step $\delta t$ & 0.12500 & 0.06250	&    0.03125  &   0.01562 & 0.00781& \cellcolor[gray]{0.8}  0.00390   \\  \hline
		    $ L^{\infty}$-error & 0.56472 & 0.25302 & 0.08831 & 0.01503 & 0.01067 & \cellcolor[gray]{0.8}  0.01866   \\ \hline
			Convergence order    &	-	&	1.15   & 1.51 & 2.55 & 0.49 & \cellcolor[gray]{0.8}  -0.80 \\ \hline
			CPU time (min)& 10 & 20 & 41 & 84 & 166 &\cellcolor[gray]{0.8}    329 \\ \hline
		\end{tabular}
	\end{center}
	\caption{Results of numerical tests for selected mesh sizes $dx$ and corresponding number of grid points $N$ in one spatial direction. Highlighted columns show results for the common value of $\delta t = 0.0039$.}
	\label{tab_errors}
	\end{table}
\begin{figure}
\centering{
	\includegraphics[width=.48\linewidth]{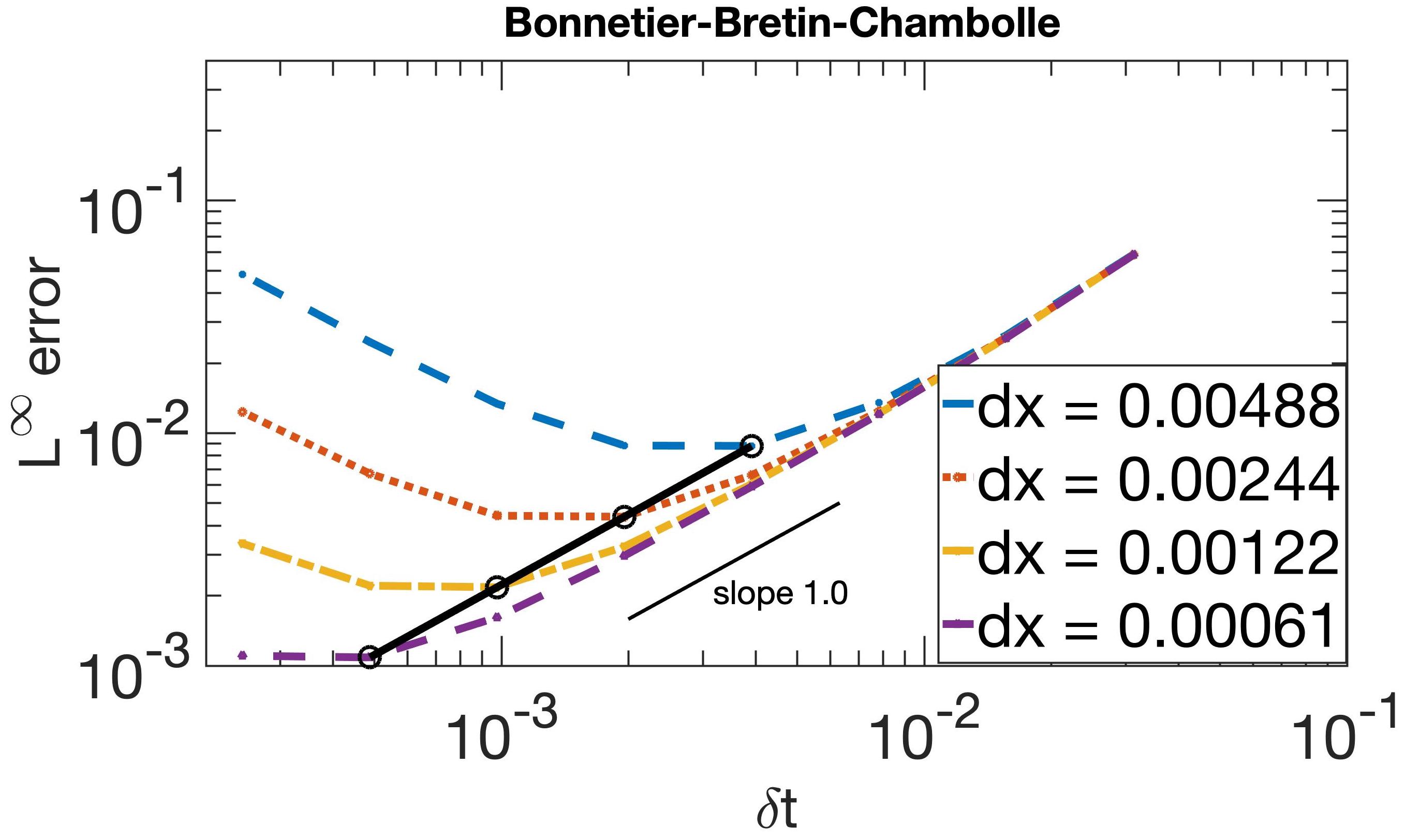}\hfill
	\includegraphics[width=.48\linewidth]{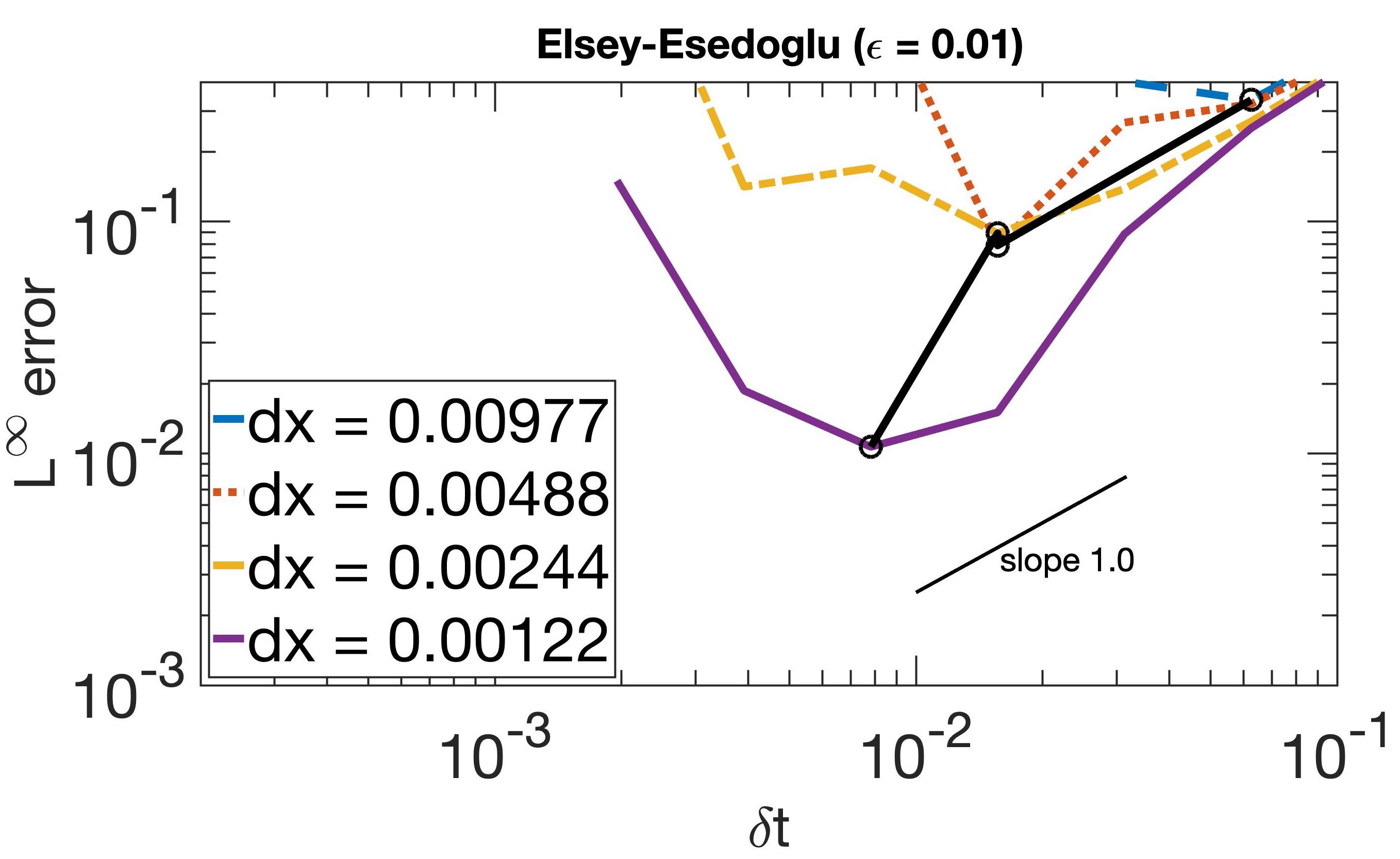} \hfill
	\includegraphics[width=.48\linewidth]{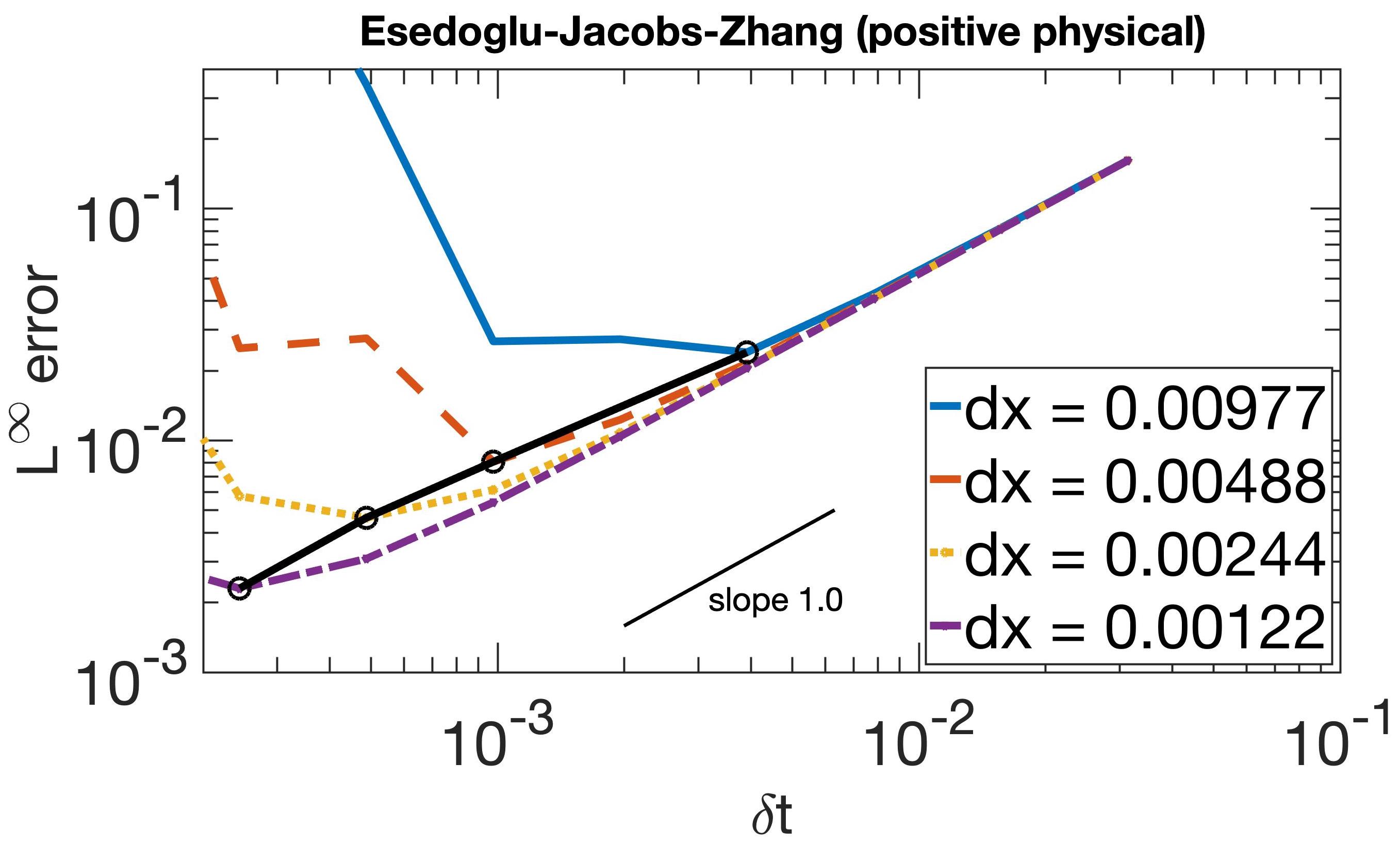}\hfill
	\includegraphics[width=.48\linewidth]{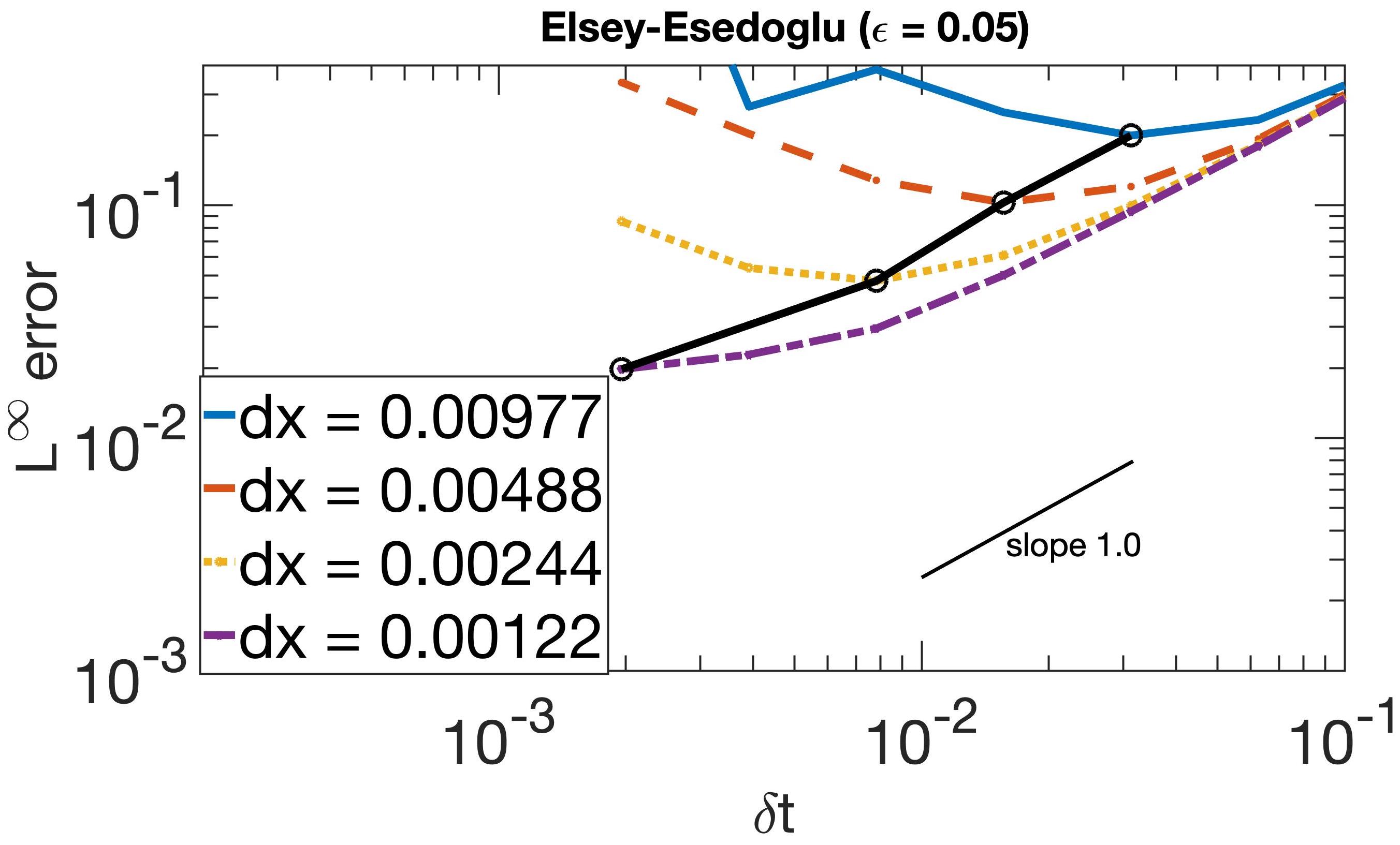} \hfill
	\includegraphics[width=.48\linewidth]{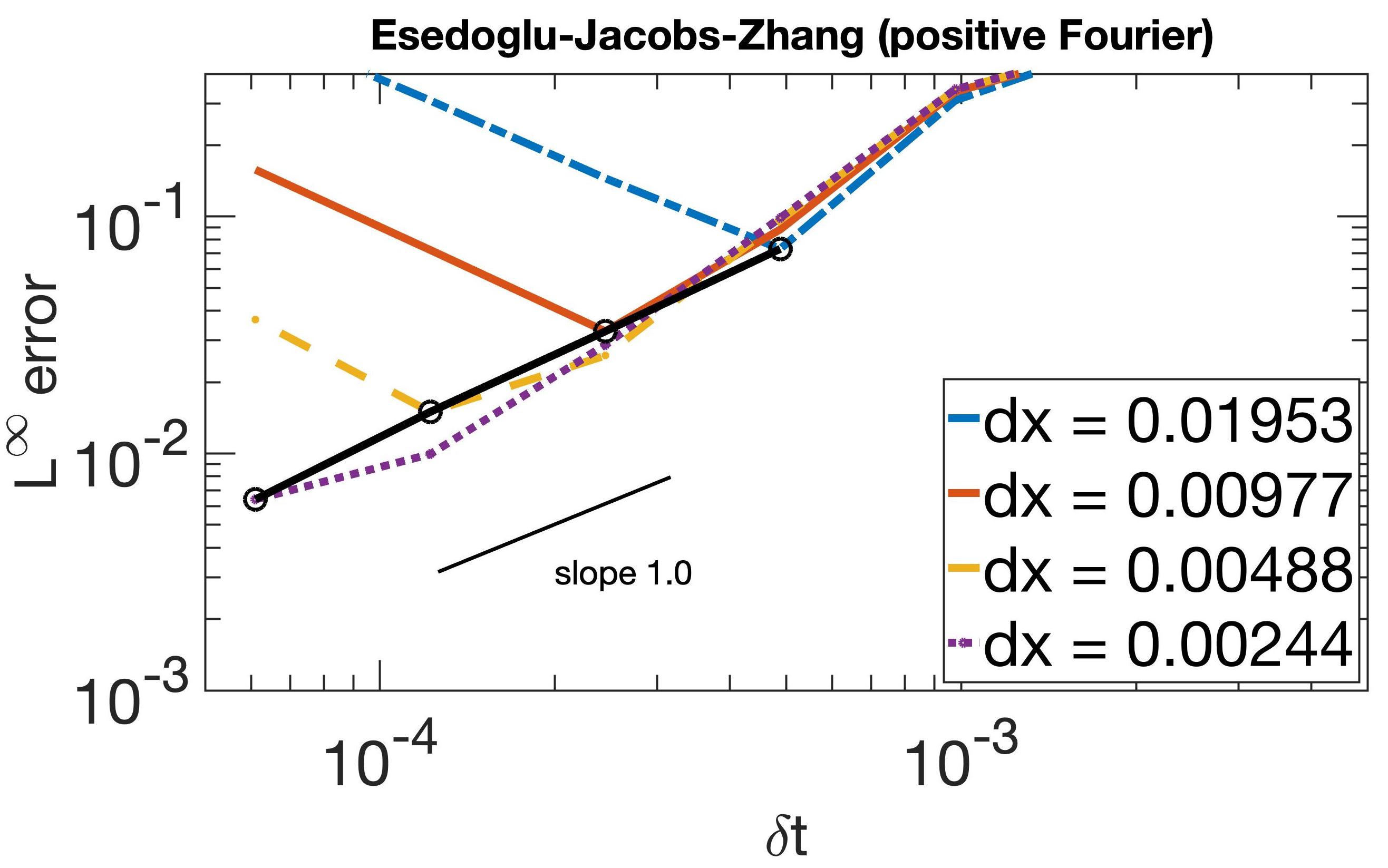}  \hfill
	\includegraphics[width=.31\linewidth]{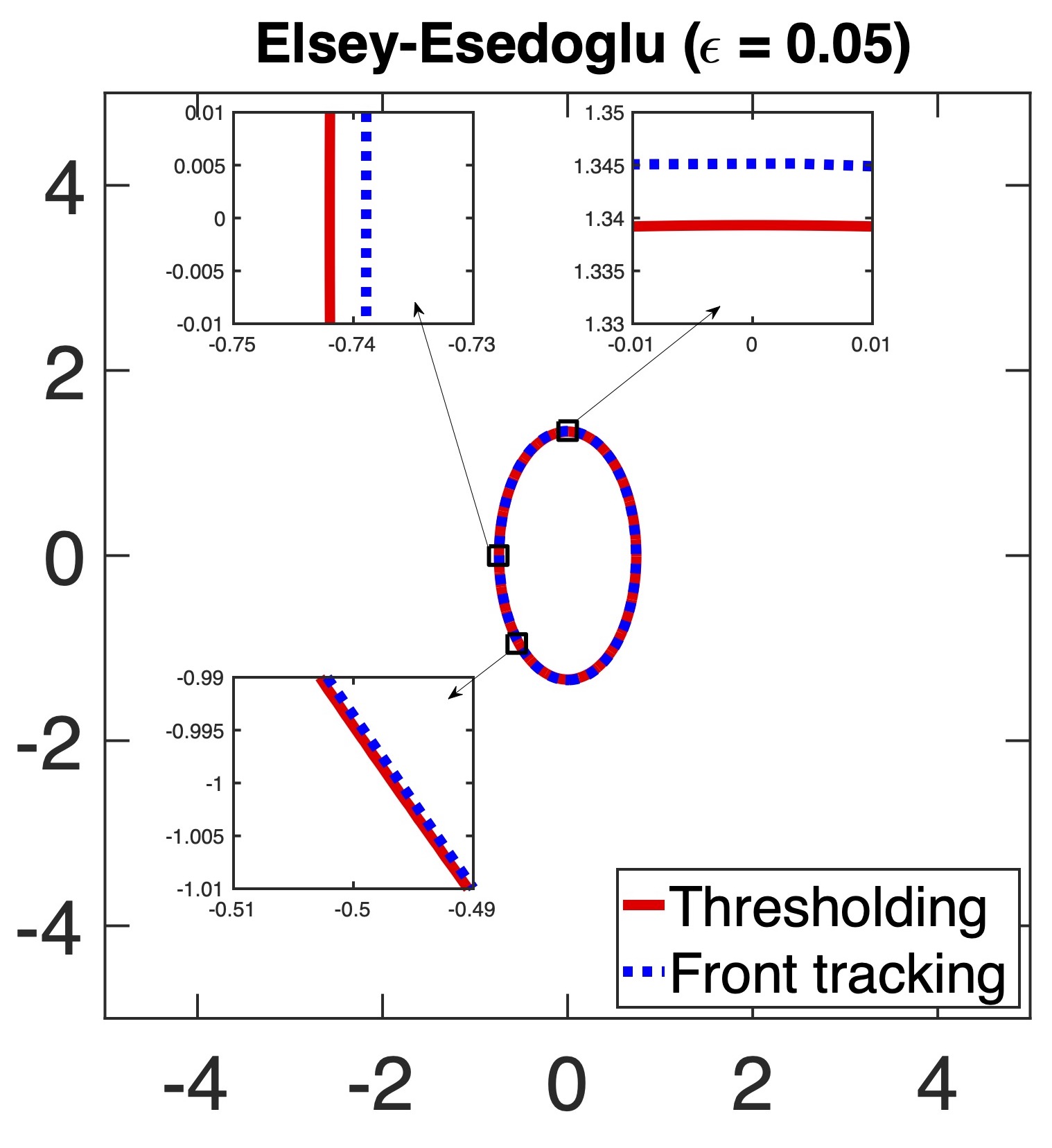} 
	\caption{Log-log plots of the dependence of numerical error on time step $\delta t$ for various mesh refinements $dx$ and for each kernel. Black points indicate the optimal $dx$-$\delta t$ pair for each value of $dx$. Figure on bottom right depicts the comparison with front-tracking solution at half-time to extinction.}
	\label{fig_convanal}}
\end{figure}

The error is taken in the $L_t^{\infty}(L_x^1)$-norm, i.e., at each time step we compute the spatial error as the $L^1$-norm of the difference of the subgrid-accurate characteristic function of the numerical solution and the subgrid-accurate characteristic function of the analytical solution, and take the maximum of this spatial error over all time steps until half-time to extinction $t= \frac{1}{4}$. Since the error tends to increase as time progresses, this maximum is usually attained at this final half-time.
The exact solution is known in analytical form for the self-similar case with mobility $\mu = \gamma$, which is available for both the BBC and EJZ kernels.
However, for EE kernels the mobility is given by \eqref{eq_mobee} and analytical solution is not known in closed form.
Therefore, we implemented a front-tracking scheme with automatic point redistribution developed in \cite{Sevcovic2011} to obtain an accurate approximation of the analytical solution (see Figure \ref{fig_convanal}, bottom right). Moreover, extinction time changes due to the different mobility, hence to get comparable results, we measured the error again at approximately half-time to extinction. 
The resulting errors, convergence rates and CPU times are presented in Table \ref{tab_errors} and Figure \ref{fig_convanal}.
{{This and all subsequent computations were}} done using Matlab software on 4 cores of 3.7GHz Intel Xeon E5 processors of a Mac Pro workstation. 
Since we only aim at relative comparison of the kernels, we did not attempt any code optimization, so the CPU times reflect the straightforward implementation of the algorithm.
Inspection of the results leads to the conclusion that the thresholding scheme possesses the expected first order convergence in time irrespective of which kernel is used.
This is observed by following the error for optimal time step $\delta t$ corresponding to a given mesh size $dx$ (black circles in Figure \ref{fig_convanal}). 
Otherwise the error dependence has a typical V-shape with errors increasing for time steps of order significantly smaller than the spatial grid size (see \cite{Misiats2016} for details).
We also note that it is computationally more efficient to have an explicit form of Fourier transform of the kernel since it saves one FFT calculation in \eqref{eq:fft}.
Nevertheless, there are a few differences between the performance of the kernels. 
For well-behaved anisotropies like the elliptic one, BBC and EJZ (positive in physical domain) kernels perform somewhat better than EJZ (positive in Fourier domain) and EE kernels. 
This can partially be observed by following the highlighted columns in Table \ref{tab_errors}, which show results obtained for time step $\delta t = 0.0039$. One confirms that decreasing value of regularizing parameter $\epsilon$ in EE kernels leads to improvement in error. Furthermore, BBC kernel yields the least error for a given $\delta t$.
Regarding EJZ kernels, the version positive in physical domain is effective with respect to CPU time.
On the other hand, since the constant $c$ in \eqref{ab2} decreased the actual mobility of EJZ kernel (positive in Fourier domain), attaining a given error required smaller time step $\delta t$, overriding this kernel's advantage of having explicit form of Fourier transform. 
The convergence order of EE kernel strongly depends on the regularization parameter $\epsilon$, namely, depending on the size of $\epsilon$ there is a value in time step $\delta t$ below which the convergence order starts deteriorating.
In addition, construction of EE kernels requires computation of a convolution at each spatial grid point which is expensive. 
In Table \ref{tab_errors}, we excluded the time needed to construct each kernel from CPU time.
We conclude that BBC kernel is superior to other kernels from the viewpoint of error, CPU time, time required to construct the kernel and regular behavior of convergence order.
\subsubsection{Other numerical tests}
Here we report on additional tests regarding the kernels' capability of dealing with sharp corners and non-convex initial conditions.

To investigate the behavior at corners, we chose the crystalline anisotropy $\gamma (x,y) = |x| + |y|$, whose Wulff shape is a square, and evolved the initial condition given as a circle by area-preserving flow, until no change was observed in the solution between subsequent time steps. Area preservation is obtained by suitably adjusting the thresholding height, in the same way as in Algorithm 2 below.
Since EJZ kernels require smooth anisotropy, the regularization $\gamma_{\epsilon} (x,y) = \sqrt{\epsilon^2 + x^2} + \sqrt{\epsilon^2 + y^2}$ with $\epsilon=0.01$ was used. 
The results are depicted in Figure \ref{fig:crystal} from two perspectives: figures on the left and in the center show the approximation of the square Wulff shape by each kernel for a fixed $dx$--$\delta t$ pair ($dx$ is the same while $\delta t$ is larger in the figure on the left) and figure on the right shows the best obtained result by each kernel for a fixed value of $dx$ (thus $\delta t$ varies according to the kernel).
One observes that BBC kernel tends to smooth out sharp corners excessively, while EE and EJZ kernels give comparable results with EE kernel showing good agreement with analytical solution when regularising parameter $\epsilon \rightarrow 0$, as already mentioned in \cite{Elsey2018}. 
The fact that a better result was obtained for EJZ kernel compared to its sibling, BBC kernel, was expected since the Fourier transform of BBC kernel for the crystalline anisotropy is singular at the origin, while this is remedied in the construction of EJZ kernel.
One can also notice that refinement of time step $\delta t$ does not necessarily lead to a better result since the interface may "get stuck" for small $\delta t$'s, exactly in the same manner as in the analysis of Figure \ref{fig_convanal}. The optimal $\delta t$ is kernel-dependent as seen in Figure \ref{fig:crystal} for BBC and EE ($\epsilon =0.01$) kernels.

\begin{figure}[h!t]
\centering{
\includegraphics[width=.30\textwidth]{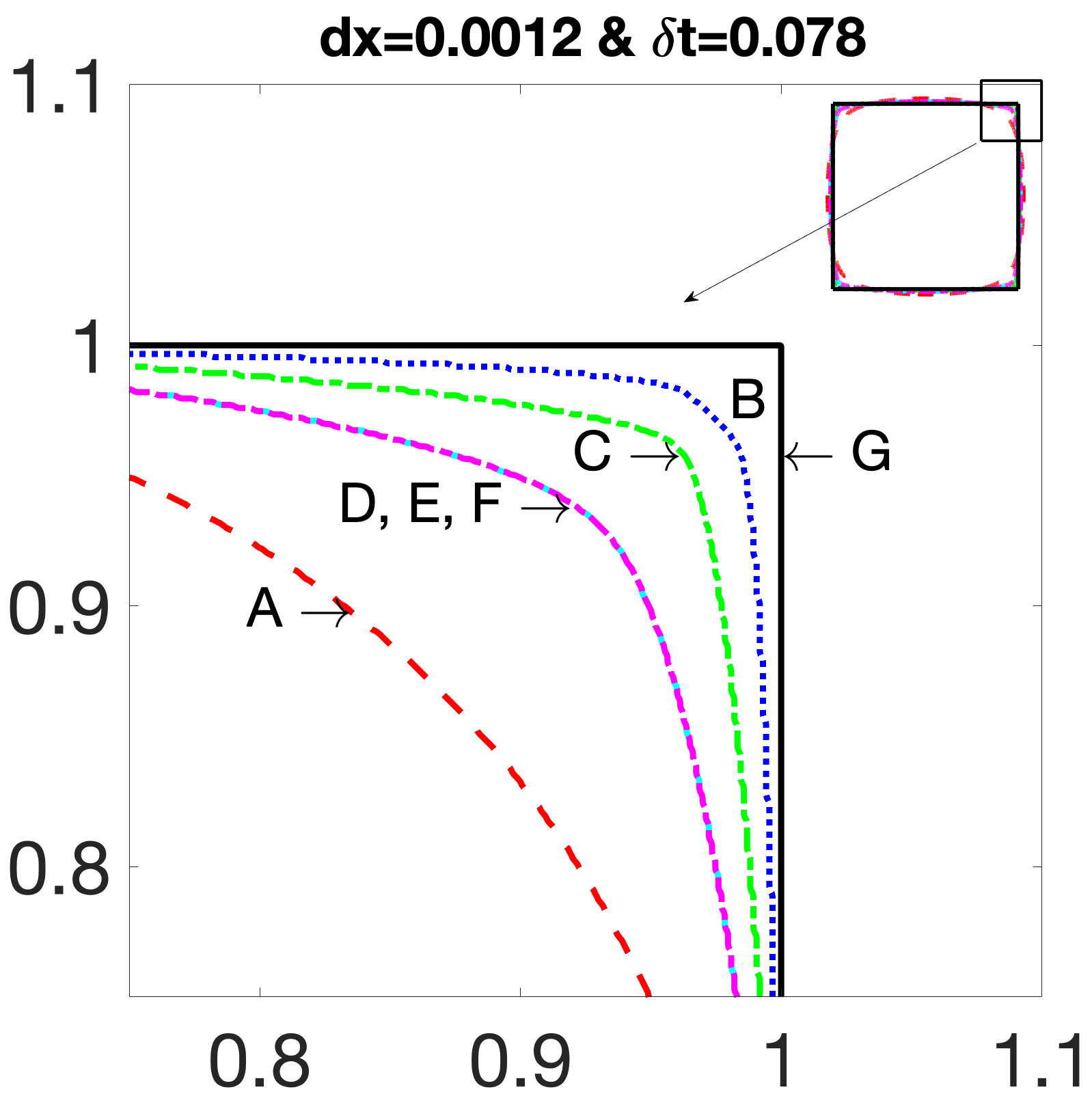}
\includegraphics[width=.32\textwidth]{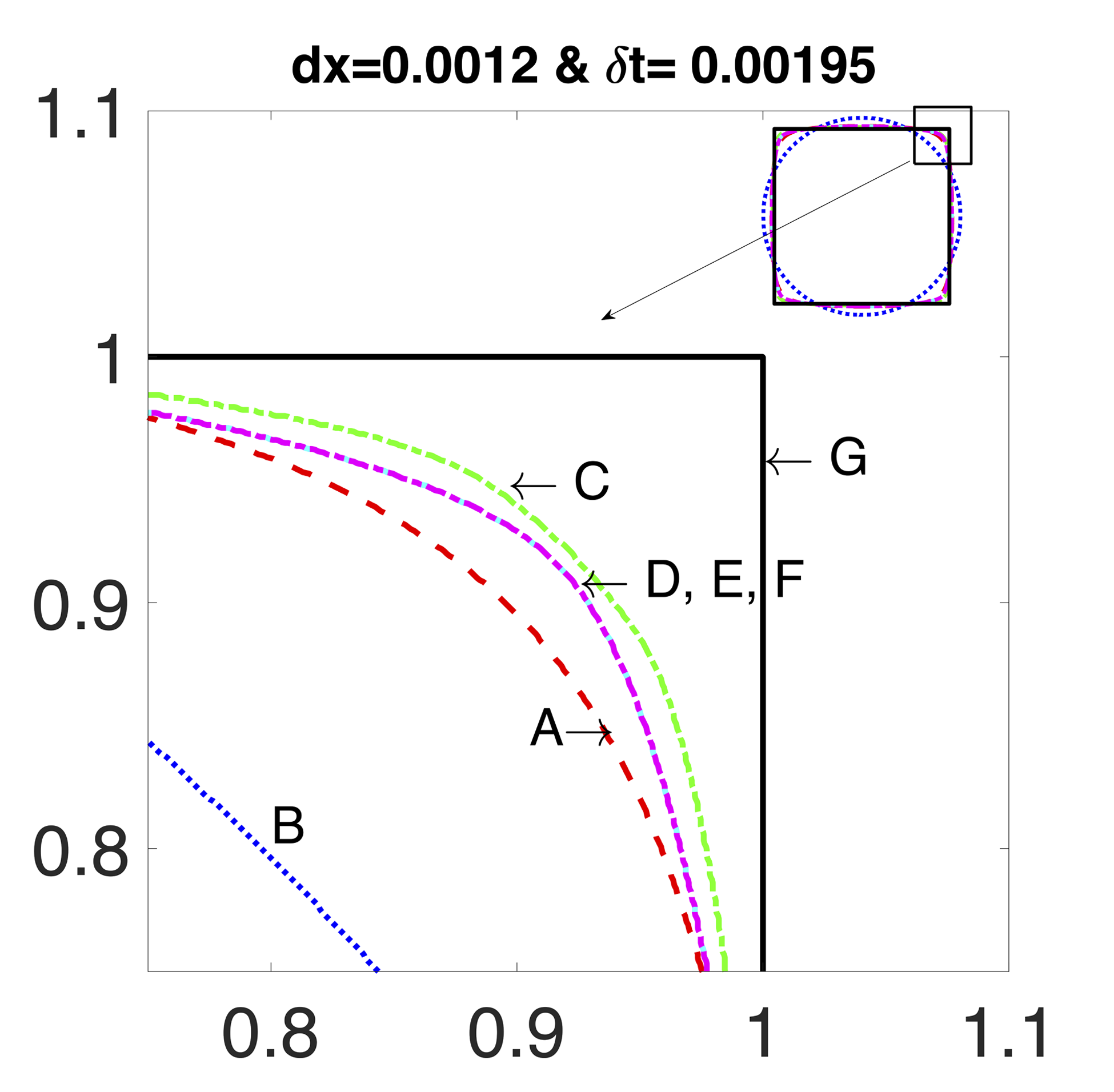}
\includegraphics[width=.32\textwidth]{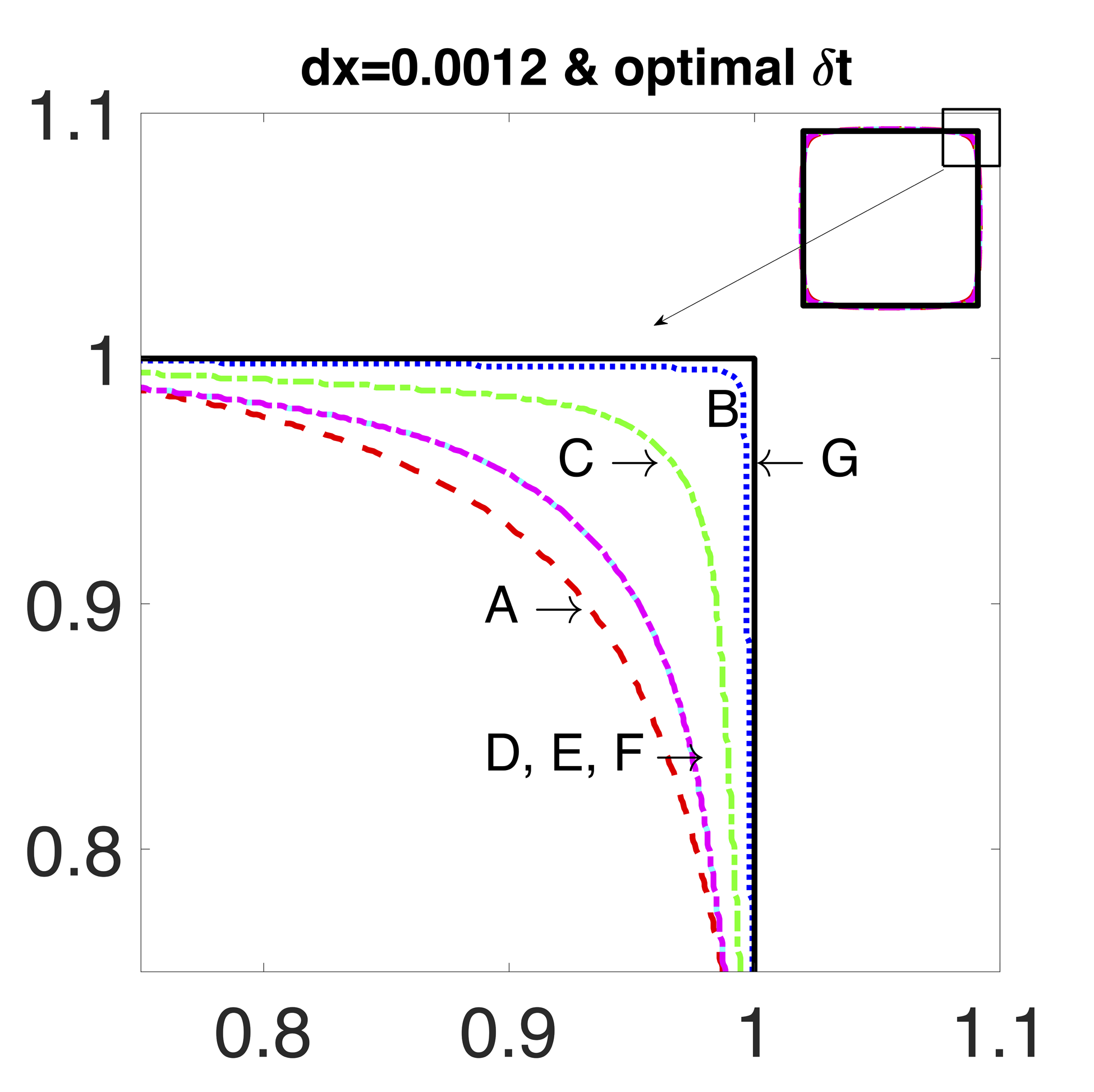}
\caption{Capability of thresholding kernels to approximate crystalline Wulff shapes. \textbf{A}: BBC kernel, \textbf{B}: EE kernel ($\epsilon=0.01$), \textbf{C}: EE kernel ($\epsilon=0.05$), \textbf{D}: EE kernel ($\epsilon=0.1$), \textbf{E}: EJZ kernel (positive in Fourier domain), \textbf{F}: EJZ kernel (positive in physical domain), \textbf{G}: Exact solution.}
\label{fig:crystal}}
\end{figure}
\begin{figure}[h!t]
\centering{
\includegraphics[width=.32\textwidth]{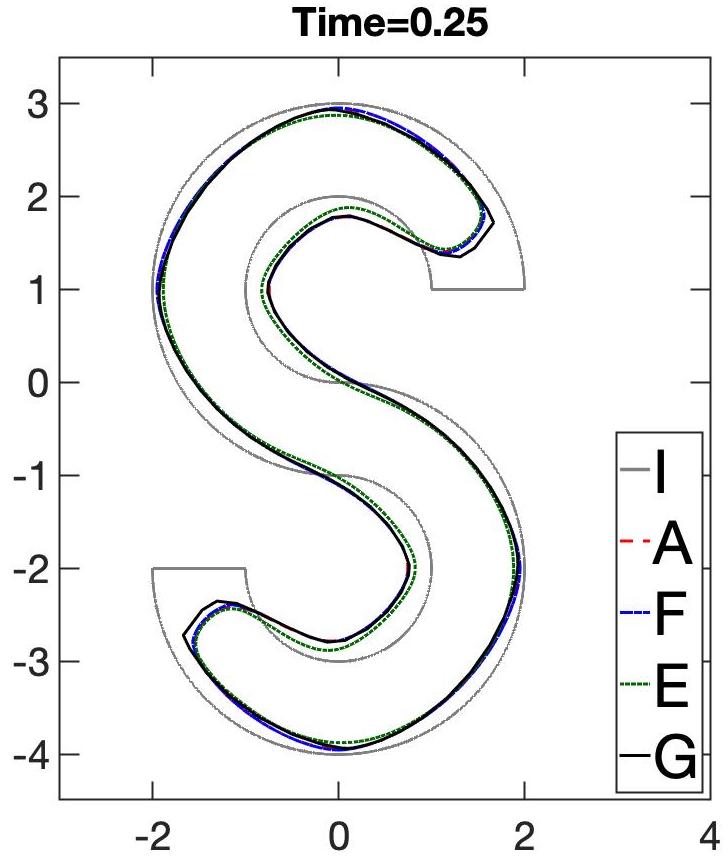}
\includegraphics[width=.32\textwidth]{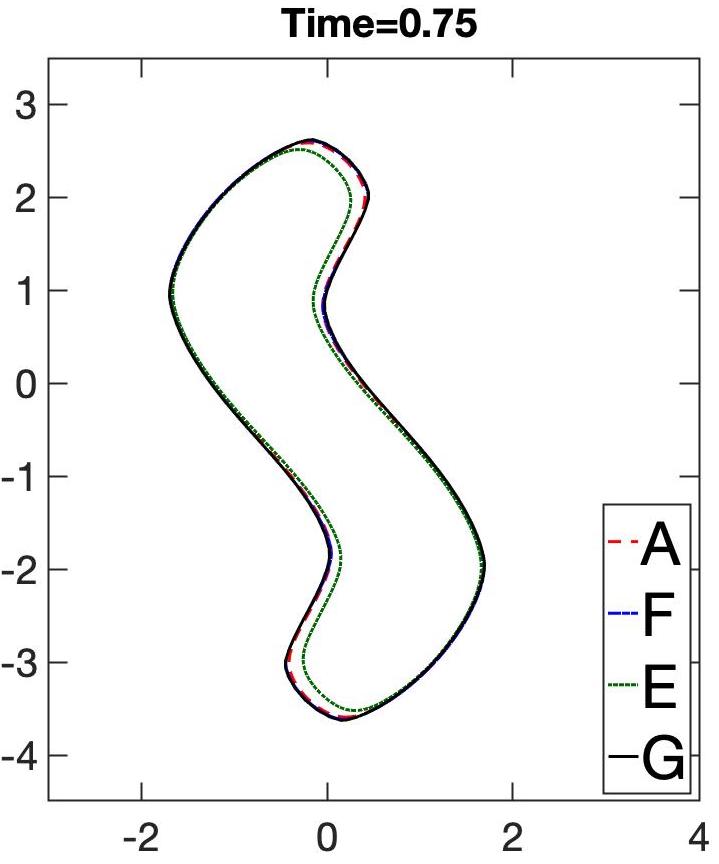}
\includegraphics[width=.32\textwidth]{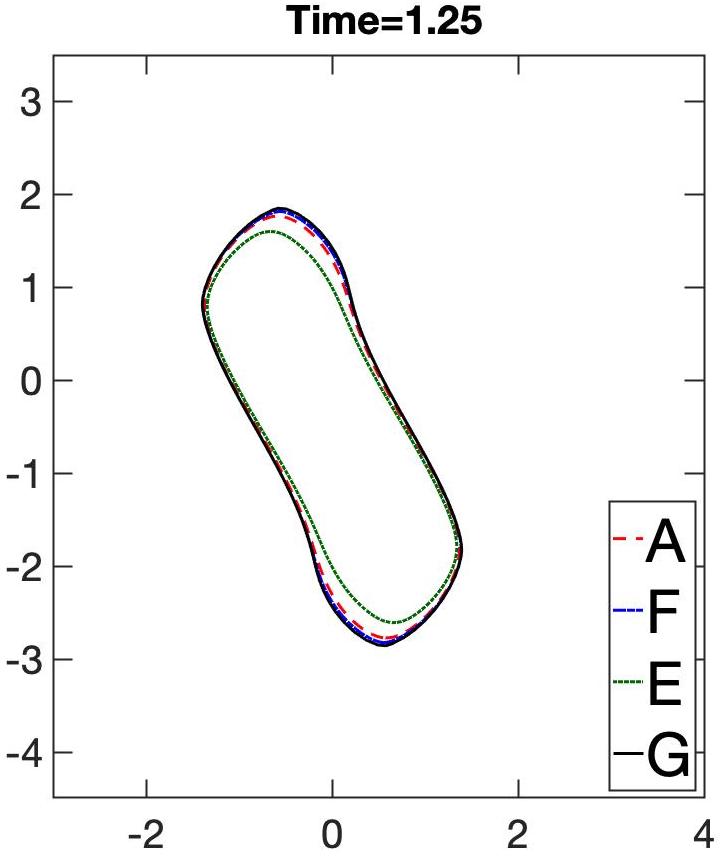}
\caption{Results of applying the thresholding method to S-shaped nonconvex initial condition. Here $dx = 0.0012$ and $\delta t =  0.000488$. \textbf{A}: BBC kernel, \textbf{E}: EJZ kernel (positive in Fourier domain), \textbf{F}: EJZ kernel (positive in physical domain), \textbf{G}: Front tracking solution, \textbf{I}: Initial shape.}
\label{fig:S shape}}
\end{figure}

The behavior of the thresholding method with nonconvex initial conditions was tested on the four-fold anisotropy $\gamma (\theta) = 1+ 0.05 \cos 4\theta$ with natural mobility $\mu = \gamma$, starting with an S-shaped initial condition shown in Figure \ref{fig:S shape} on the left.
Numerical solution was compared against an accurate approximation of the analytical solution obtained by the anisotropic front-tracking method with automatic point redistribution \cite{Sevcovic2011}.
Since mobility cannot be prescribed for EE kernels, we tested only BBC and EJZ kernels, which yielded a good match with the front-tracking solution {obtained by discretizing the curve into 102 points}, differences being caused mainly by smoothening out of sharp corners by the kernels (see Figure \ref{fig:S shape}). 

\section{Threshold dynamics with obstacle}

In this section we develop and analyze a thresholding algorithm for motion of anisotropic interfaces on obstacles, with the view of application to the problem of particles evolving on substrates.

\subsection{Derivation and stability}

First we derive the thresholding algorithm and prove its stability.
We closely follow the ideas in \cite{Esedoglu2015} and \cite{Xu2017} for the isotropic case.
The setup of the problem is the same as in Section \ref{sec_model}: A particle occupying a region $P$ is evolving within a domain $\Omega$ on a flat substrate given by the region $S$.
We replace $\mathbb{R}^d$ by a domain $\Omega$ which is sufficiently large but bounded, for two reasons: to avoid infinite value of energy and to make numerical implementation feasible. 
According to \eqref{eq:limlya}, the total energy \eqref{eq:energy} can be approximated by
\begin{equation}
\label{eq:aproxener}
E_{\delta t} (P) = \frac{1}{\sqrt\delta t}  \int_{\Omega}^{} (
\one_{P} K_{\delta t} \ast \one_{V}  + 	\gamma_{SP} \one_{P} G_{\delta t} \ast \one_{S} + 	\gamma_{SV} \one_{S} G_{\delta t} \ast \one_{V})\, dx,
\end{equation}	
where $K$ is a suitable kernel representing the  anisotropy $\gamma_{PV}$. The function $K$ is assumed to be symmetric, sufficiently smooth and (integrally) positive definite on $\Omega$, i.e., \begin{equation*}
\int_{\Omega} K(x,y)u(x)u(y)\, dx \, dy \geq 0 \qquad \text{for any } \; u\in L^1(\Omega).
\end{equation*} 
Recall that all kernels introduced in Section \ref{sec_aniker} satisfy these conditions if the Wulff shape corresponding to the anisotropy $\gamma = \gamma_{PV}$ is centrally symmetric, convex and smooth.
Since  the surface tensions  $\gamma_{SP}$ and $\gamma_{SV}$ are constant, Gaussian kernel $G_{\delta t} (x) = \frac{1}{(4 \pi \delta t)^{d/2}} e^{\frac{- |x|^{2}}{4 \delta t}} $ is used for the approximation of the corresponding surface energies \cite{BMO1994}.
Note that since the substrate does not deform, the evolution of the vapor region $V$ is fully determined by that of the particle region $P$, which is why the approximate energy is written only as a function of $P$. In the sequel, this dependence will be equivalently expressed by the corresponding characteristic functions: $u=\one_P$, $\one_V=\one_{\Omega^\text{up}}-u$, i.e., we will write $E(u)$ to mean $E(P)$, where $u=\one_P$. Recall that $\Omega^\text{up}=\overline{\Omega \setminus S}$ is the part of the domain $\Omega$ occupied by particle and vapor.

We wish to minimize \eqref{eq:aproxener} among characteristic functions $u=\one_P$ with a given integral $\int_{\Omega} \one_P \, dx =A$, which form a non-convex set. In the following lemma, this is relaxed to a mathematically more amenable convex constraint, by allowing $u$ to take any value between $0$ and $1$.

\begin{lem}
	\label{lemma_1}
	If the kernel $K$ is smooth and positive definite, $\mathcal{L} (u) $ is a linear functional and $\alpha$, $\beta$ are non-negative real numbers, then
	to minimize $\alpha E_{\delta t} (u) + \beta \mathcal{L} (u)$ over the non-convex set
	$$ \mathcal{B} = \{ u  \in BV (\Omega^{\text{up}}) | \; u (x)  \in \{0,1\} \; \text{a.e.} \; x \in \Omega^\text{up}, \; \int_{\Omega^\text{up}} u\, dx = A  \} $$
	is equivalent to the minimization of the same functional over the convex set
	$$ 	\mathcal{K} = \{ u \in BV (\Omega^{\text{up}}) | \; u (x) \in [0,1] \; \text{a.e.} \; x \in \Omega^\text{up}, \; \int_{\Omega^\text{up}} u\, dx = A  \}. $$
	{{Here $BV$ denotes the space of functions with bounded variation.}}
\end{lem}
\begin{proof}
The existence of minimizers in both $\mathcal{B}$ and $\mathcal{K}$ can be proved by the direct method when the kernel $K$ is smooth.
Since $\mathcal{B} \subset \mathcal{K}$, it is enough to prove that if $\widetilde{u} \in \mathcal{K}$ is a minimizer of $ \alpha E_{\delta t} (u) + \beta \mathcal{L} (u)$ in $\mathcal{K}$ then $\widetilde{u} \in  \mathcal{B}$. 
This is clear when $\alpha = 0$ since the minimizer of a linear functional over a convex set must belong to the boundary of the set. 

For $\alpha > 0$, we use contradiction. Assuming $\widetilde{u} \not \in  \mathcal{B}$, we deduce the existence of a measurable set $Z \subset \Omega^{\text{up}}$ with positive measure and of a constant $c \in (0, \frac{1}{2})$ such that 
$$ 0 < c < \widetilde{u}(x), \;\; \one_{\Omega^\text{up}}(x) - \widetilde{u}(x) < 1-c <1 \qquad \forall x \in Z. $$
We partition the set $Z$ into two disjoint subsets $Z_1$ and $Z_2$ of equal measure $|Z|/2$, and define $u^t = \widetilde{u} + t \one_{Z_1} - t \one_{Z_2}$.
Note that $\int_{\Omega} u^t \, dx = \int_{\Omega} \widetilde{u} \, dx = A$, and that $0 \leq u^t \leq 1$ holds for $t \in (0,c)$, so $u^t \in\mathcal{K}$ for such $t$. Hence, direct computation yields
\begin{multline*}
	\frac{d^2}{dt^2} \left( \alpha E_{\delta t} (u^t) + \beta \mathcal{L} (u^t) \right) \\ =  \frac{\alpha}{\sqrt{\delta t}} \int_{\Omega^\text{up}} \frac{d}{dt} u^t K_{\delta t} \ast  \frac{d}{dt} (-u^t)\, dx
	=   \frac{-\alpha}{\sqrt{\delta t}} \int_{\Omega^\text{up}} (\one_{Z_1} - \one_{Z_2}) K_{\delta t} \ast (\one_{Z_1} - \one_{Z_2}) \, dx.
\end{multline*}
Due to positive definiteness of the kernel $K$, this value is negative, which implies that $\widetilde{u}=u^t|_{t=0}$ cannot be a minimizer. 
\end{proof}

The above lemma implies that, the minimization of \eqref{eq:aproxener} 
can be done over the relaxed set $\mathcal{K}$ without changing the result.
We solve this minimization problem by iterations. Assume we have an approximation $u^k \in \mathcal{B}$, where $k$ is step number during iteration. 
Energy functional $E_{\delta t} (u)$ linearized at the point $u^k$ reads
$$  E_{\delta t} (u) = E_{\delta t} (u^k) + \mathcal{L} (u -u^k, u^k) + \text{higher order terms},$$
where
\begin{multline*}
\mathcal{L} (v, u^k ) = \frac{1}{\sqrt{\delta t}}  \int_{\Omega^\text{up}} v \left( K_{\delta t} \ast (\one_{\Omega^\text{up}}-u^k ) + \gamma_{SP} G_{\delta t} \ast  \one_{S} \right) \, dx \\ 
- \frac{1}{\sqrt{\delta t}} \int_{\Omega^\text{up}} v \left( K_{\delta t} \ast  u^k + \gamma_{SV} G_{\delta t} \ast  \one_{S} \right) \, dx.
\end{multline*}
The idea then is to minimize the linearized functional:
\begin{eqnarray}
\label{eq_lindf}
\min_{u \in \mathcal{K}} \mathcal{L} (u, u^k),
\end{eqnarray} 
and take the minimizer as an improved approximation $u^{k+1}$ to the minimizer of $E_{\delta t}$. 
Lemma \ref{lemma_1} guarantees  that the solution of \eqref{eq_lindf} belongs to $\mathcal{B}$. 
The following lemma shows that \eqref{eq_lindf} can be solved by a simple thresholding. 
\begin{lem}
	\label{lemma_2_algorithm}
	\normalfont
		Write $\mathcal{L}(v,u^k)$ as $\int_{\Omega^{\text{up}}} v \phi^k \, dx$, where
		\begin{equation*}
		\phi^k(x) = \frac{1}{\sqrt{\delta t}} \left( K_{\delta t} \ast (\one_{\Omega^\text{up}}-2u^k ) + (\gamma_{SP} -\gamma_{SV}) G_{\delta t} \ast  \one_{S}   \right),
		\end{equation*}
	and let
	$$ P^{k+1}= \{ x \in \Omega^\text{up} | \; \phi^k(x) < \delta^k  \}  \quad  \text{and} \quad V^{k+1} = \Omega^\text{up} \setminus P^{k+1},	$$ 
	where $\delta^k$ is chosen in such a way that the area of the particle is preserved, i.e., so that $\int_{\Omega^\text{up}} \one_{P^{k+1}} \, dx = A$. Then $u^{k+1} = \one_{P^{k+1}}$ is a solution to \eqref{eq_lindf}.
\end{lem}
\begin{proof}
We wish to prove
$ \mathcal{L} (u^{k+1}, u^k ) = \min_{u \in \mathcal{K}} \mathcal{L} (u, u^k)$, which is, by Lemma \ref{lemma_1}, equivalent to
\begin{eqnarray*}
	\mathcal{L} (u^{k+1}, u^k)  \leq \mathcal{L} (u, u^k) \qquad \forall u \in \mathcal{B} .
\end{eqnarray*}
Every element of $\mathcal{B}$ looks like $u = \one_{R}$ for some open set $R \subset \Omega$ such that $|R| = A$. 
Denote $D_1 = R \setminus P^{k+1}$ and $D_2 = P^{k+1} \setminus R$. Due to the area constraint, $D_1$ and $D_2$ satisfy $|D_1| = |D_2|$. Note that
\begin{eqnarray*}
  && \phi^k (x) \geq  \delta^k \quad \forall x \in D_1 \subset \Omega^\text{up}\setminus P^{k+1} \qquad \text{ and}\\
  && \phi^k (x) <  \delta^k \quad \forall x \in D_2 \subset P^{k+1} .
\end{eqnarray*}
Using these inequalities we calculate
\begin{multline*}
	\mathcal{L}(u^{k+1}, u^k)  - \mathcal{L}(u, u^k) =  \int_{\Omega^\text{up}}  (u^{k+1} - u) \phi^k   \, dx
	= \int_{D_1} (-\phi^k) \, dx + \int_{D_2} \phi^k \, dx\\
	\leq \int_{D_1} (-\delta^k)\, dx + \int_{D_2}  \delta^k\, dx 
	= \delta^k (-|D_1| + |D_2|) = 0,
\end{multline*}
reaching the desired conclusion.
\end{proof}

We arrive at the thresholding Algorithm 2. 
	\begin{algorithm}[ht]
		\caption{Evolution of anisotropic particle on substrate}
					Given phases $P^k$, $V^k \subset \Omega^\text{up }$ at time $t_k = k\cdot \delta t$, to get new phases $P^{k+1}$ and $V^{k+1}$ at next time step $t_{k+1} = (k+1 )\delta t$, perform the following two steps:
					\begin{eqnarray*}
					&&\text{Convolution :}	 \quad \phi^k(x) = \tfrac{1}{\sqrt{\delta t}} \left( K_{\delta t} \ast (\one_{\Omega^\text{up}}-2u^k ) + (\gamma_{SP} -\gamma_{SV}) G_{\delta t} \ast  \one_{S}   \right), 
					\\
					&&\text{Thresholding:}	 \hspace*{0.5cm} P^{k+1}  =  \Big\{x \in \Omega^{\text{up}}|  \; \phi^{k} (x) < \delta^k \Big\},\;\; V^{k+1}  =  \Omega^\text{up}  \setminus P^{k+1}.
					\end{eqnarray*}
					Here $\delta^k$ is chosen {{so that the area of phase $P^{k+1}$ is  equal to the area $A$ of phase $P^k$}}.
	\end{algorithm}
	
Next we show that this algorithm is stable, that is, the total energy $E_{\delta t}(P^k)$ is a non-increasing function of $k$.
\begin{thm}\label{thm_stability}
		\normalfont
		Set $u^k =\one_{P^{k}}$ for $k = 0,1,2,...$, where $P^k$ is obtained in Algorithm 2. Then 
		\begin{eqnarray}
		E_{\delta t} (u^{k+1}) \leq E_{\delta t} (u^k) \qquad \text{for all } \; k = 0,1,2,\dots\;  \text{ and all } \; \delta t >0.
		\end{eqnarray}
\end{thm}
\begin{proof}
Definitions of $E_{\delta t}$ and $\mathcal{L}$ and Lemma \ref{lemma_2_algorithm} yield
\begin{equation}
E_{\delta t} (u^k ) - \frac{1}{\sqrt{\delta t}}\int_{\Omega^\text{up}} \left( u^k K_{\delta t} \ast u^k  + \gamma_{SV} G_{\delta t} \ast \one_S \right) \, dx  = \mathcal{L} (u^{k}, u^k)
 \geq \mathcal{L} (u^{k+1},u^k). \label{thmproof1}
\end{equation}
Furthermore,
\begin{multline}
 \mathcal{L} (u^{k+1},u^k) = E_{\delta t} (u^{k+1} ) - \frac{1}{\sqrt{\delta t}}     \int_{\Omega^\text{up}} \bigg( u^{k+1} K_{\delta t} \ast u^k  
 +  u^{k} K_{\delta t} \ast u^{k+1} \\ 
 -  u^{k+1} K_{\delta t} \ast u^{k+1} +\gamma_{SV} G_{\delta t} \ast \one_S \bigg) \, dx \label{thmproof2}
\end{multline}
Gathering \eqref{thmproof1} and \eqref{thmproof2} leads to
 \begin{equation}
 \label{thmproof3}
E_{\delta t} (u^{k+1})  \leq E_{\delta t}  (u^k)+ \mathcal{Y} ,
 \end{equation}
 where
 \begin{eqnarray}
 \nonumber
 \mathcal{Y}& =&   \frac{1}{\sqrt{\delta t}}  
  \int_{\Omega^\text{up}} \left( u^{k+1} K_{\delta t} \ast u^k  + u^{k} K_{\delta t} \ast u^{k+1}  -  u^{k+1} K_{\delta t} \ast u^{k+1} - u^{k} K_{\delta t} \ast u^k \right) \, dx \\ 
  \nonumber
 & = & - \frac{1}{\sqrt{\delta t}}    
 \int_{\Omega^\text{up}} (u^{k+1} -u^k) K_{\delta t} \ast (u^{k+1} -u^k) \, dx  \\
 &\leq& 0 .
 \label{thmproof4}
 \end{eqnarray}
 \eqref{thmproof3} and \eqref{thmproof4} finish the proof of the theorem.
\end{proof}
The convergence of Algorithm 2 to the $L^2$-gradient flow of the energy $E$ as $\delta t\to 0+$ was proved in \cite{Laux2016} 
in the isotropic setting without obstacle but remains open for the anisotropic case and for the problem with obstacle.

\subsection{Numerical tests}
\subsubsection{Convergence analysis}
\label{sec:particletestorig}
\begin{figure}
\centering{
	\includegraphics[width=.48\linewidth]{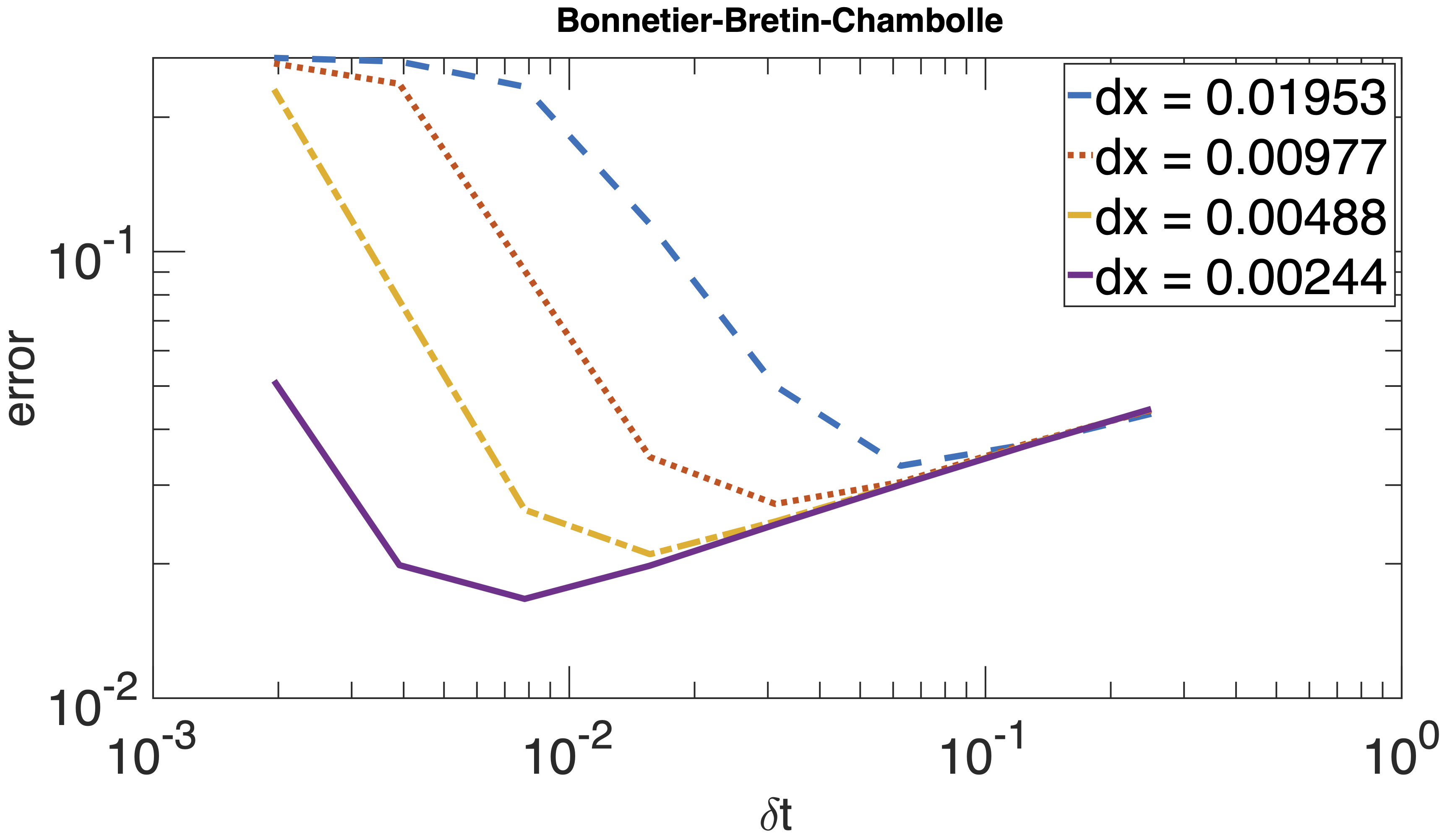}\hfill
	\includegraphics[width=.48\linewidth]{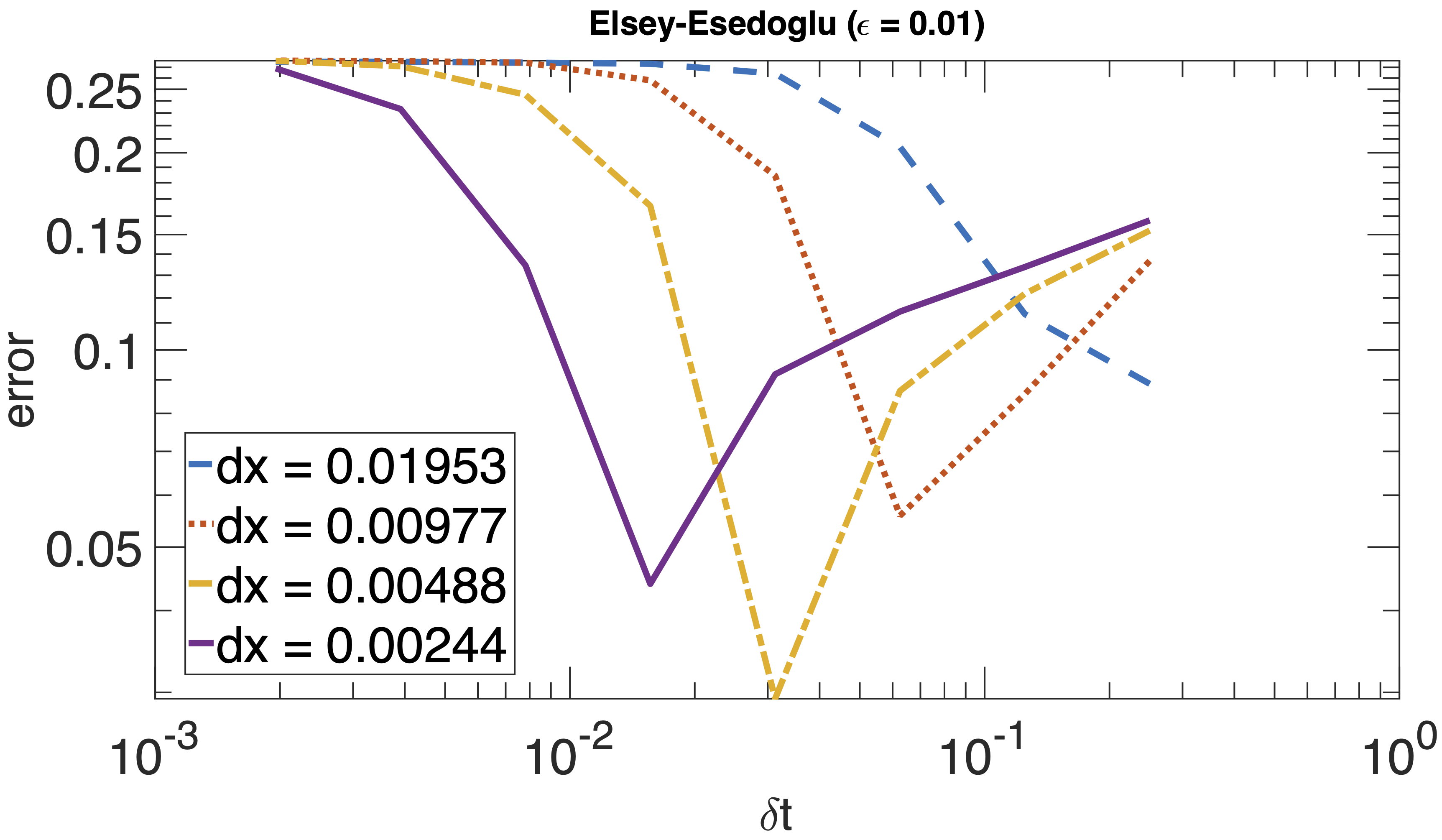}\hfill
	\includegraphics[width=.48\linewidth]{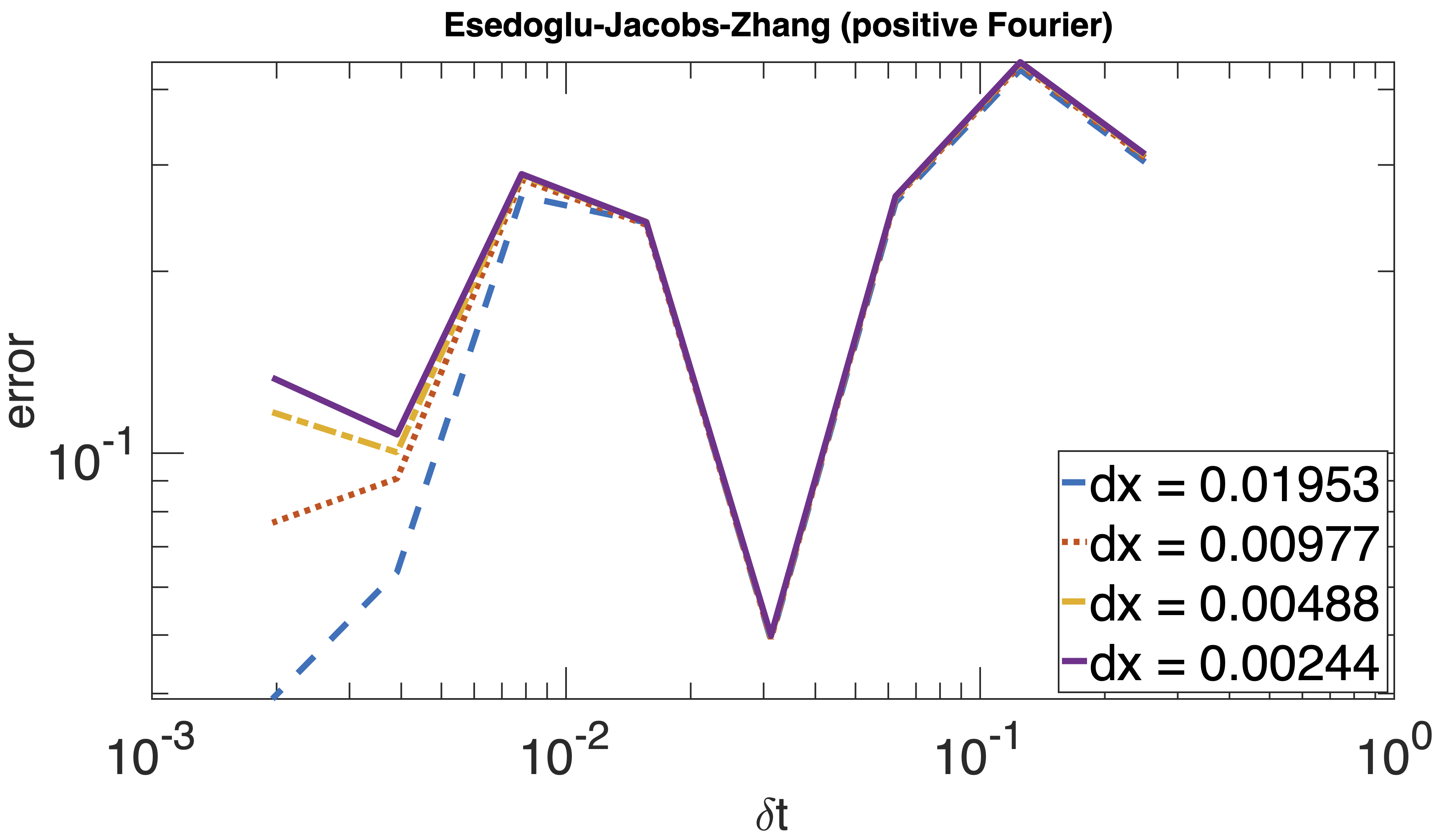}\hfill
	\includegraphics[width=.48\linewidth]{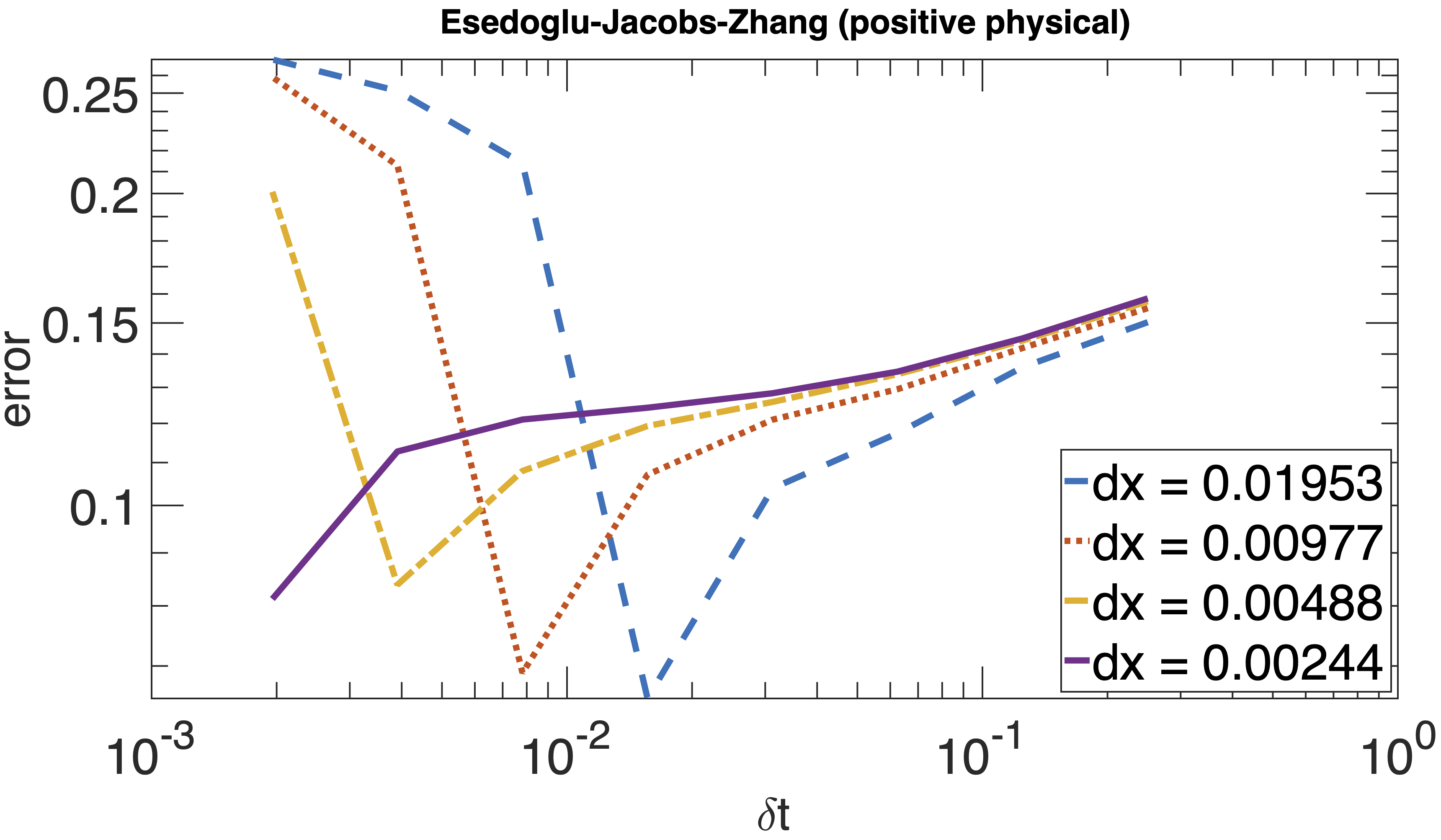}\hfill
	\includegraphics[width=.48\linewidth]{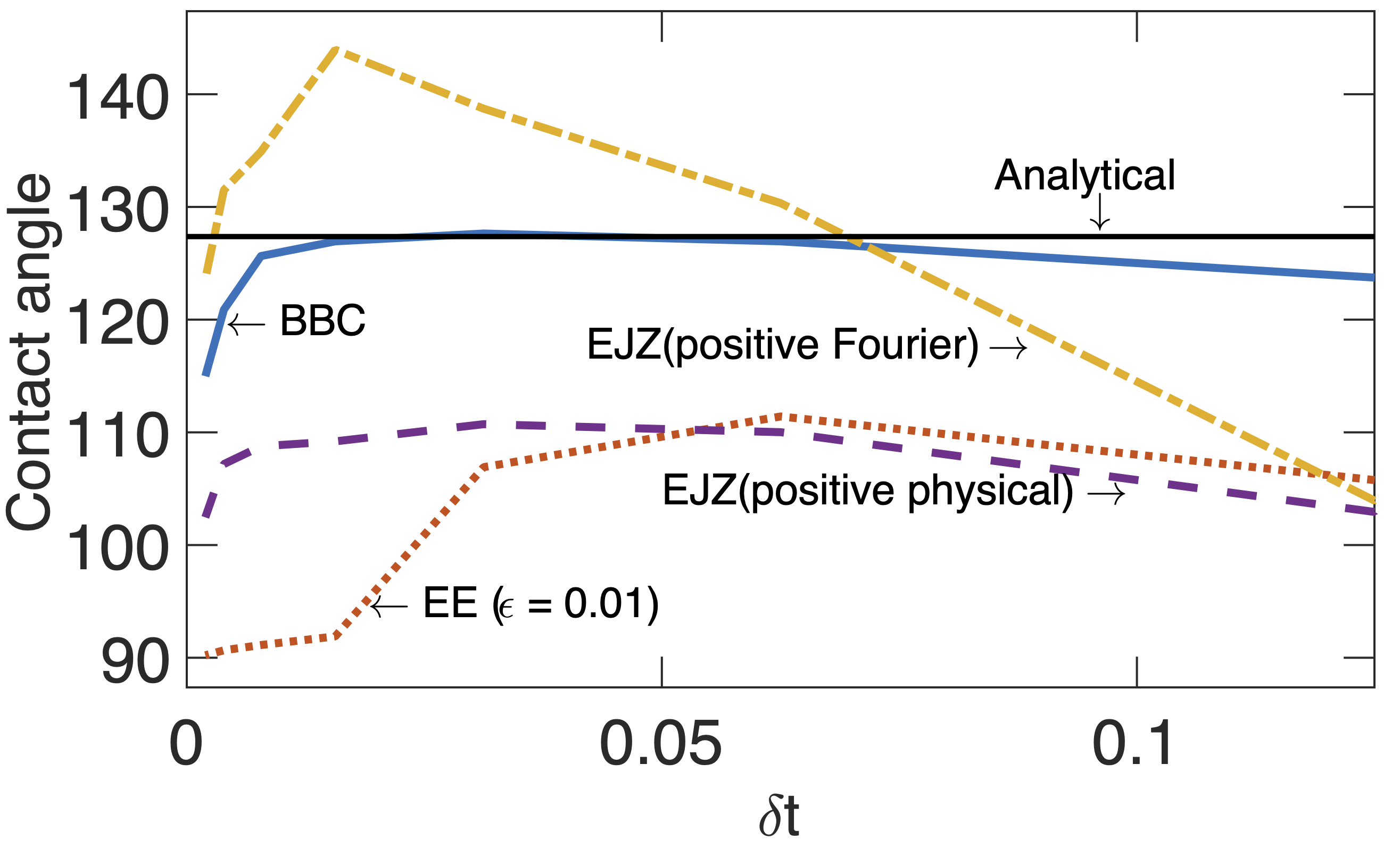} \hfill
	\includegraphics[width=.4\linewidth]{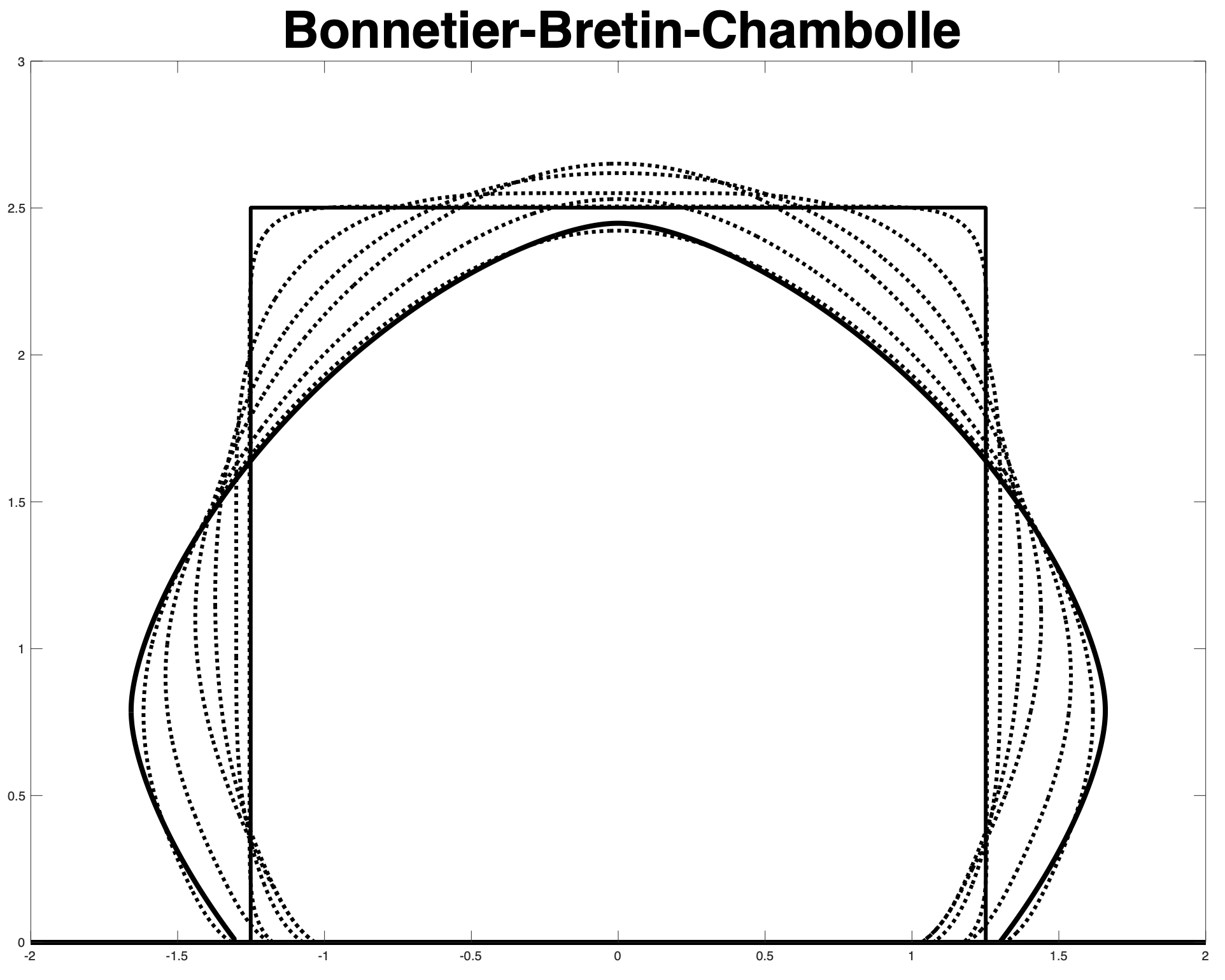} \hfill
	\caption{Evolving particle on substrate with anisotropies $\gamma = 1 + 0.05 \cos 4\theta$,  $\gamma_{SP}= 1.5$ and $\gamma_{SV}= 1$. Four figures on top show log-log plots of the dependence of relative error in shape on time step $\delta t$ for various spatial mesh resolutions $dx$. (Bottom left) Numerical contact angle versus time step $\delta t$ for {$dx=0.0098$}. (Bottom right) Evolution of particle by BBC kernel at times $0, 10\delta t, 25\delta t,  50\delta t,  192\delta t \text{ and }672\delta t$, for $dx=0.0024$ and $\delta t = 0.0078$; solid line represents analytical solution.}
\label{fig_particle_all_kernels}}
\end{figure} 

Evolving particle on substrate is a three-phase problem and thus a new issue that was not present in the two-phase setting of Section \ref{sec:twophase} is the realization of correct contact angles at the triple point where the three phases meet. 
Therefore, in numerical tests, besides the error measuring the discrepancy in the shape of the interface, we focus on the quality with which thresholding Algorithm 2 approximates the exact contact angles.
The evaluation is carried out based on the stationary solution of the area-preserving anisotropic mean curvature flow on obstacle, for which analytical solution can be obtained through Winterbottom construction \cite{Winterbottom1967} and contact angles by solving anisotropic Young equation \eqref{eq:aniyoung}.

The setup of the simulations is as follows. 
A flat substrate is positioned at height $y=0$ of the computational domain $\Omega =[-5,5]\times [-5,5]$ and the initial shape of the particle is the square $[-1.25,1.25]\times [0,2.5]$. The particle is equipped with four-fold anisotropy $\gamma (\theta) = 1+ 0.05 \cos 4\theta$, while the other anisotropies are set to $\gamma_{SP} = 1.5$ and $\gamma_{SV}=1$.
The corresponding stable shape of the particle is depicted at bottom right in Figure \ref{fig_particle_all_kernels}.
We take the mobility as $\mu=\gamma$, except for the EE kernel, where the mobility is determined by anisotropy.
The initial shape is evolved by Algorithm 2, which is expected to approximate the evolution laws \eqref{eq:firstvar} and \eqref{eq:aniyoung}, until it reaches the numerical stationary state, i.e., there is no change between two subsequent time steps.
Area preservation, i.e., determination of the value of $\delta^k$ in Algorithm 2, is implemented by sorting values of $\phi^k$ at grid points in increasing order and accepting first $M$ grid points in $P^{k+1}$, where $M$ is the number of grid points in the initial particle. This approach is simple and preserves number of grid points in $P^k$ exactly but is possible only thanks to fixing the substrate -- it would not work for a three-phase evolution, where all phases can deform.

We executed Algorithm 2 for each kernel and for several combinations of spatial mesh size $dx$ and time step $\delta t$, and measured the error as the area of the symmetric set difference of exact and numerical stationary solutions divided by the area of the particle. Furthermore, numerical contact angle was obtained by fitting a linear function to several points representing numerical interface near the substrate, excluding points in the immediate vicinity of the substrate. The dependence of resulting values of contact angles on the range of points selected for fitting is negligible, as long as the fitting points do not extend too far from the substrate. Since symmetry was preserved by the scheme, we present contact angle at the left triple point only. 

Figure \ref{fig_particle_all_kernels} shows the obtained results. Upon inspection, one sees that EJZ kernel (positive in physical domain) shows the weakest convergence response, both in overall shape and contact angle. We remark that the non-monotone jumpy behavior of the error may also be caused by the absence of subgrid accuracy improvement in these simulations. However, the irregular behavior of the EJZ error seems to be inherent to the kernel itself, 
as already hinted at in \cite{Esedoglu2017b}.
Surprisingly, the EE kernel exhibits systematic convergence but the size of the error is relatively large compared to other kernels, both in overall shape and contact angle, in spite of choosing its smoothing as $\epsilon=0.01$, which gave relatively good results in the two phase case. This could be due to oscillatory behavior observed for most of $dx$--$\delta t$ combinations in the numerical solution obtained by EE kernels near the stationary state.
Among the investigated kernels, the BBC kernel behaves in the most regular way. Moreover, BBC gives the least value of error among all the kernels and also the best agreement in contact angle.

\subsubsection{Modified algorithm to improve contact angle}
Proposition \ref{prop:stability} asserts that the thresholding algorithm always decreases total energy over time in the space-continuous setting. However, as was already pointed out in the numerical analysis of Section \ref{sec:twophase}, when space is discretized, smaller $\delta t$ does not necessarily entail improved quality of the approximation. 
In particular, for large values of $\delta t$, the approximation of evolution by the algorithm improves upon reducing $\delta t$ up to a certain threshold value, which depends on spatial mesh size $dx$ and which we call "optimal $dx$--$\delta t$ pair". Further decrease in $\delta t$ then leads to deteriorating error, and for sufficiently small $\delta t$ the evolution stagnates entirely. This is due to the fact that for excessively small $\delta t$ relative to mesh size $dx$, the interface is not able to proceed more than one spatial grid size in the diffusion step and the thresholding step returns it to its original position. As discovered in Section \ref{sec:conv2phase} the optimal $\delta t$ for a given $dx$ is of the order $dx$.
On the other hand, the dynamics near triple points is known to occur on the scale of $\sqrt{\delta t}$ \cite{Esedoglu2015}. 
Consequently, to get an accuracy of order $\delta t$ including triple points, one would need to take a spatial mesh of size $\delta t^2$, so that stagnation of interfaces is avoided.

This problem was addressed in \cite{Xu2017} by introducing time step scaling. 
The authors start with a relatively large time step $\delta t$ and compute the evolution until the interface does not evolve anymore. 
Then the time step is decreased, e.g., by halving, and the evolution, especially around triple points, is refined. 
Since the interfaces away from junctions are already in their right position when the time step halving starts, the stagnation phenomenon has smaller impact on the outcome.
We extend the time halving scheme to anisotropic energies, which leads to Algorithm 3, where the difference between two solutions at different time steps is evaluated based on a threshold $\tau >0$ for the area of symmetric set difference.

	\begin{algorithm}[ht]
		\caption{Anisotropic particle on substrate (modified for higher contact angle accuracy)}
					{{Given initial phases $P^0$, $V^0 \subset \Omega^\text{up }$, time step $\delta t^0$ and a threshold $\tau$, to generate phases $P^{k}$ and $V^{k}$ at time steps $t_{k+1} = t_{k} + \delta t^{k}$, $k = 0,1,2, \dots$, set $P^{*}=P^0$ and repeat the following steps:}}
					\begin{eqnarray*}
					&&\text{Convolution :}	 \quad \phi^k(x) = \tfrac{1}{\sqrt{\delta t^k}} \left( K_{\delta t^k} \ast (\one_{\Omega^\text{up}}-2u^k ) + (\gamma_{SP} -\gamma_{SV}) G_{\delta t^k} \ast  \one_{S}   \right), 
					\\
					&&\text{Thresholding:}	 \quad P^{k+1}  =  \Big\{x \in \Omega^{\text{up}}|  \; \phi^{k} (x) < \delta^k \Big\},\;\; V^{k+1}  =  \Omega^\text{up}  \setminus P^{k+1}.\\
					&&\text{Time scaling: }	  \cdot \text{ If }  |P^{k} - P^{k+1}| > \tau, {{\text{set} \; \delta t^{k+1} = \delta t^k \; }} \text{and go to next convolution step.}\\
				 	&& \hspace*{2.7cm} \cdot \text{ If }  |P^{k} - P^{k+1}| \leq \tau  \text{ and } |P^{\ast} - P^{k+1}| \geq \tau, \text{ set } \delta t^{k+1} = \delta t^k/2, P^{\ast} = P^{k+1},\\
				 	&& \hspace*{3.7cm}  \text{ and go to next convolution step.} \\
				 		&& \hspace*{2.7cm} \cdot \text{ If }  |P^{k} - P^{k+1}| \leq \tau  \text{ and } |P^{\ast} - P^{k+1}| < \tau, \text{ terminate. } \end{eqnarray*}
Here $\delta^k$ is chosen {{so that the area of phase $P^{k+1}$ is  equal to the area $A$ of phase $P^0$}}.
	\end{algorithm}

For numerical tests, we keep the setup and method of error measurement of previous subsection, except for the surface energies, which are chosen to pose a challenge to the numerical scheme as $\gamma (\theta) = 1 + 0.05 \cos(4\theta + 8)$, $\gamma_{SP}= 1$ and $\gamma_{SV} = 1.1$. The stable shape, shown in Figure \ref{fig:Tilted}, is then a tilted four-fold Wulff shape with one of its corners close to the substrate.
Since the kernels tend to smooth out corners, this is expected to lead to a badly resolved contact angle.
The choice of threshold $\tau$ has an impact on the final result. Our analysis shows that moderately small values of $\tau$ relative to the total area of the particle perform better, so in the simulations here we picked $\tau = 0.0001A$, where $A$ is the particle's area.
Further decrease in $\tau$ leads to improved contact angles but the resolution quality of particle's overall shape deteriorates.

\begin{figure}[htbp]
\centering{
	\includegraphics[width=.48\linewidth]{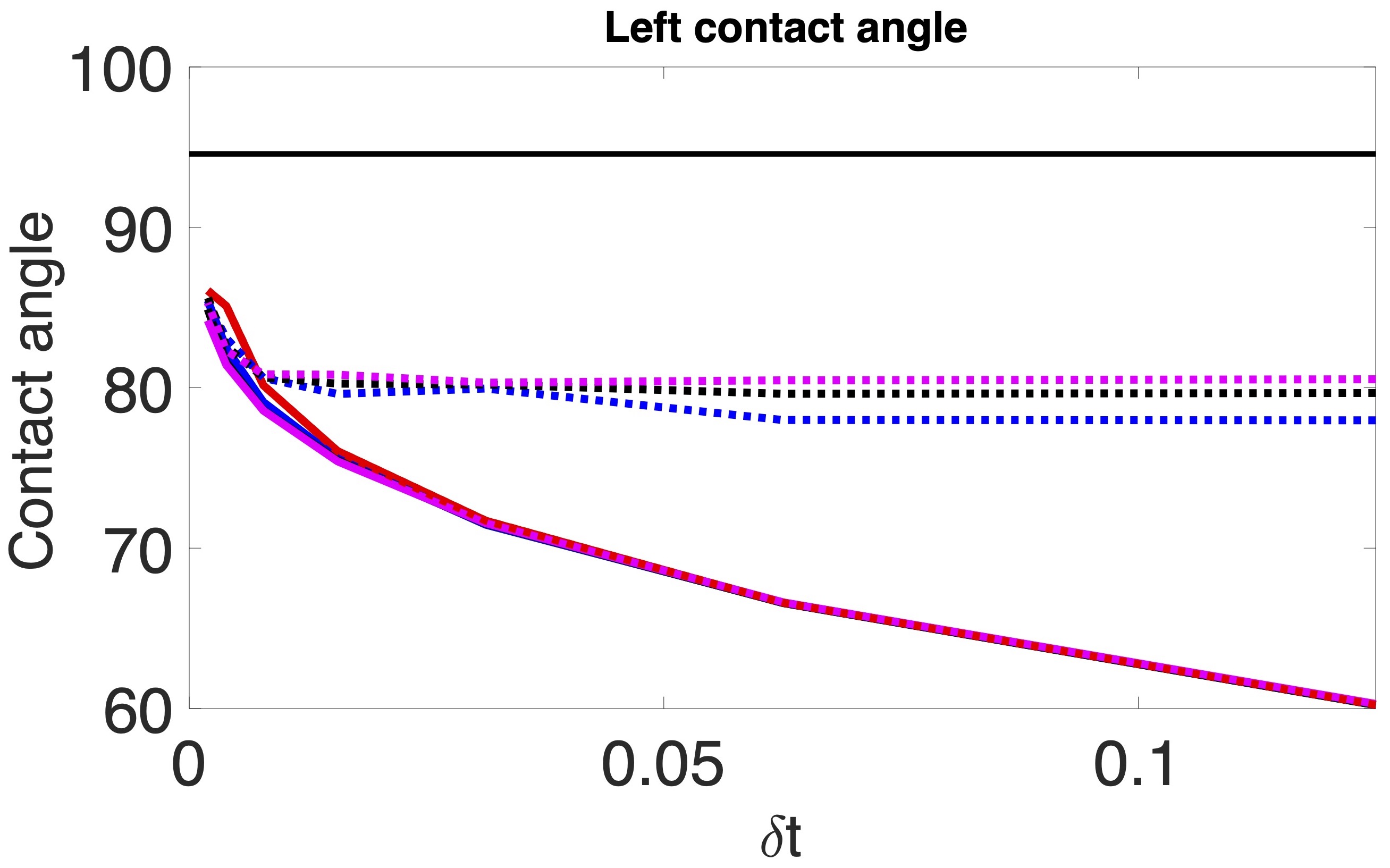}\hfill
	\includegraphics[width=.48\linewidth]{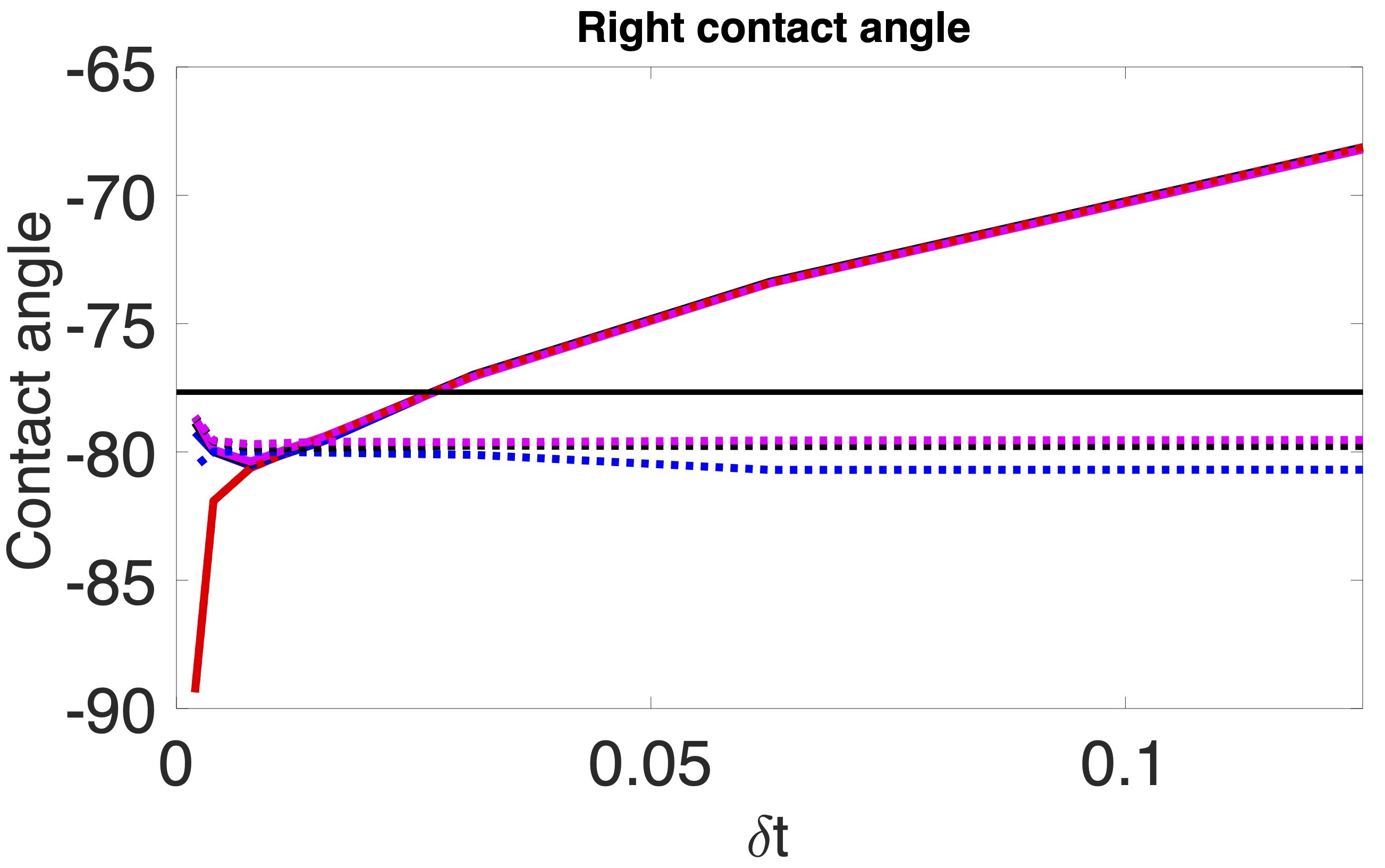}\hfill
	\includegraphics[width=.48\linewidth]{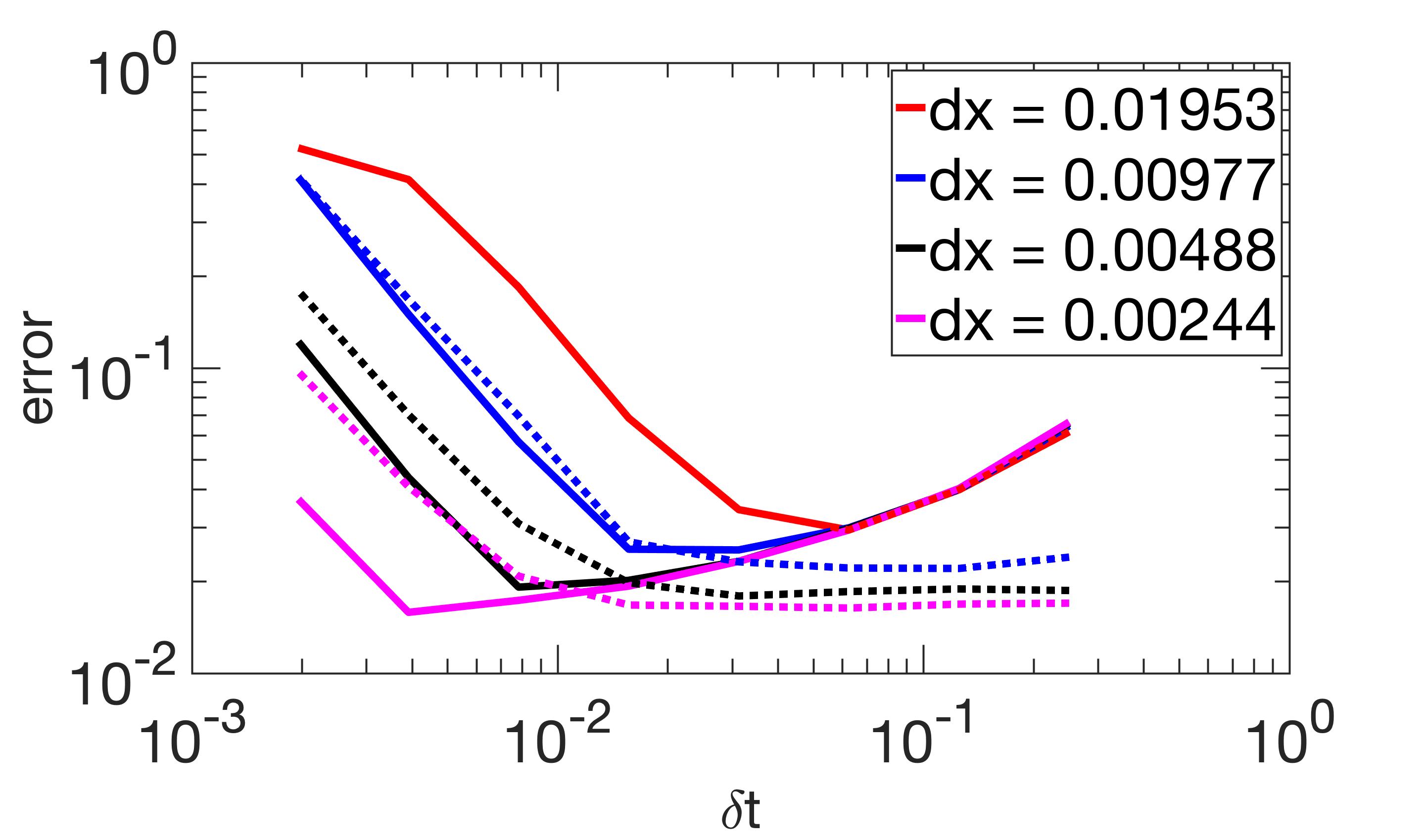}\hfill
   \includegraphics[width=.48\linewidth]{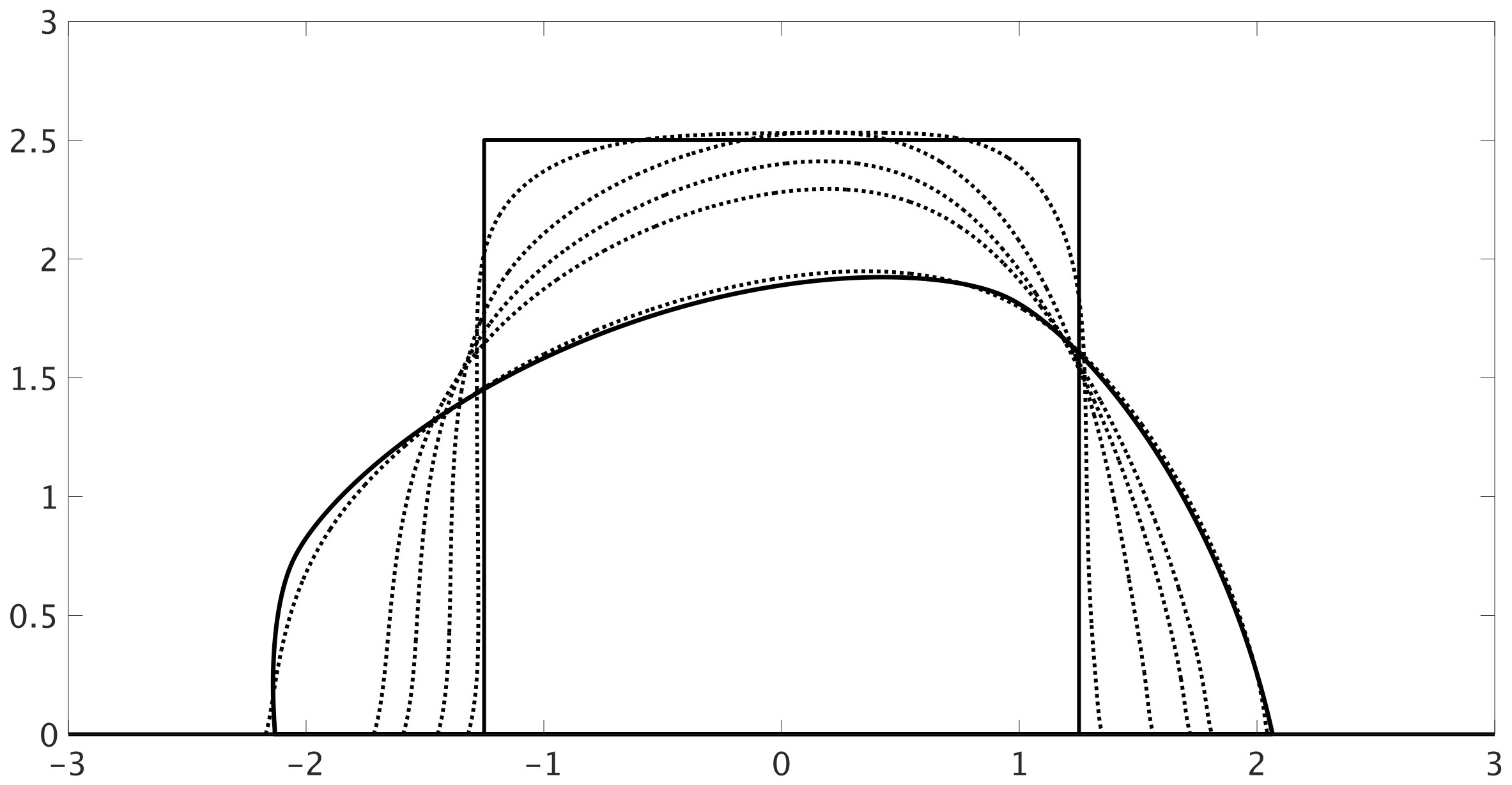}\hfill
	\caption{Analysis of Algorithm 3 modified for contact angle improvement (dotted lines) in comparison to the original Algorithm 2 (solid lines). A tilted four-fold anisotropy $\gamma (\theta) = 1 + 0.05 \cos ( 4\theta + 8)$, and  $\gamma_{SP}= 1$, $\gamma_{SV}= 1.1$ are adopted. Figures on top show the obtained contact angles, where the black lines signify the analytical contact angles $94.58^{\circ}$ (left), $-77.67^{\circ}$ (right). Figures in the second row show log-log plots of error in shape for several values of $dx$, and the computed evolution of particle at times $0, 10\delta t,  50\delta t,  100\delta t, 150\delta t \text{ and }616\delta t$, with $dx=0.0024$ and $\delta t = 0.0078$ (here solid line represents analytical stationary solution).}
	\label{fig:Tilted}}
\end{figure}
	\begin{table}[htbp] 
	\begin{center}
	\scriptsize
			\noindent
		\begin{tabular}{|c|c|c|c|c|} 
			\hline 
Time step $\delta t$& Error in area & Contact angle (left)  & Contact angle (right) & CPU time  \\ \hline
			0.2500 &   0.0163 [0.0670] &   81.15 [51.87] &  -79.03 [-60.34] & {{ 54 [18]}} min  \\ \hline
		0.1250    &	0.0161  [0.0404]	&	81.15 [60.30]  & -79.40 [-68.23]&  {{ 50 [51]}} min \\\hline
			0.0625  &  0.0158 [0.0297] &  81.23 [66.64] & -79.06 [-73.45] &  {{ 58 [69]}} min  \\ \hline
			0.03125  & 0.0157 [0.0234]  &  80.85 [71.54]  &   -79.49 [-77.08]  &  {{ 79 [102]}} min\\ \hline
				0.01562  & 0.0161 [0.0192]  &  80.87 [75.40]  &  -79.54 [-79.40] &  {{ 119 [190]}}min \\ \hline
		\end{tabular}
			\caption{Contrast between modified Algorithm 3 and original Algorithm 2 [shown in brackets]. Here $dx=0.00122$, and the correct contact angles are $94.58^{\circ}$ (left) and $-77.67^{\circ}$ (right).}
		\label{tab_errors_mod}
		\end{center}
	\end{table}
	
A comparison of outputs of Algorithm 2 and Algorithm 3 in this setup are presented in Figure \ref{fig:Tilted} and Table \ref{tab_errors_mod} for BBC kernel, which performed best in the tests of Section \ref{sec:particletestorig}.
It leads to the conclusion that, provided that initial time step $\delta t$ is chosen sufficiently large, modified Algorithm 3 yields an output which is independent of the initial $\delta t$ and corresponds to the optimal solution of the original Algorithm 2 for the given mesh size $dx$.
{{Interestingly, in spite of giving better outputs, computational time required by the modified Algorithm 3 is shorter than in the original Algorithm 2. This is caused by the termination criterion since when time step $\delta t$ is sufficiently decreased by halving, the interface stagnates leading to termination of computation. On the other hand, time step in the original algorithm is fixed and interface changes, although only slightly, over a longer time span.}}

It means that the modified algorithm is preferable since it finds the optimal solution for a given spatial mesh with a reasonably large initial time step $\delta t^0$ without negative impact on the computational cost.
On the other hand, from the error plot in Figure \ref{fig:Tilted} one observes that the modification does not improve the error of the optimal solution of Algorithm 2. This may be due to the fact that the interface away from triple point is in the stagnation mode when the neighborhood of triple point is being refined, and hence the improvement near contact point cannot be reflected in the overall shape of the particle.
This is also highlighted in the large error in the contact angle at left triple point. 
EE kernels were also tested on this setup but the time step halving modification failed to have an effect on the outcome due to the oscillations observed already in Section \ref{sec:particletestorig}.

\subsubsection{Topological changes}
The ability of handling topological changes automatically is a major advantage of level set methods.
Most of previous work on solid state dewetting problem 
uses explicit representation of the interface, i.e., front-tracking or finite elements, which requires an extra ad-hoc numerical surgery when particles merge or split.

Here we present two examples of simulations involving topology change of particles moving on substrate, namely splitting and merging.
In order to devise an experiment that leads to splitting, we consider a patterned substrate, where the effective surface tension $\gamma_S^{\text{split}} := \gamma_{SP} -\gamma_{SV}$ depends on the position on the substrate. For simplicity, we use an analogous pattern $\gamma_S^{\text{merge}}$ for the merging simulation. Specifically, we set
\begin{equation*}
    \gamma_S^{\text{split}} (x) = \left\{ \begin{array}{ll} \gamma_{S_2} = 2 \;\; & \text{for} \; x \in (-0.5,0.5) \\ \gamma_{S_1} = 0 \;\; & \text{otherwise}  \end{array} \right., \quad \gamma_S^{\text{merge}} (x) = \left\{ \begin{array}{ll} \gamma_{S_2} = -0.1 \;\; & \text{for} \; x \in (-0.5,0.5) \\ \gamma_{S_1} = 0 \;\; & \text{otherwise}  \end{array} \right.
\end{equation*}
In the thresholding algorithm, when computing the convolution $G_{\delta t} * \one_{S}$,  these values are extended to the substrate region as constants in the normal direction, as shown in Figure \ref{fig:Topological change}, that is, we replace $(\gamma_{SP} - \gamma_{SV}) G_{\delta t} * \one_{S}$ by $\gamma_{S_1} G_{\delta t} * \one_{S_1} + \gamma_{S_2} G_{\delta t} * \one_{S_2}$.
We consider two-fold anisotropy $\gamma (\theta) = 1 + 0.3 \cos(2\theta + \pi)$ 
and set up the initial condition as a rectangle for the splitting experiment and as two right-angled triangles at distance 0.4 apart for the merging experiment (see solid lines in Figure \ref{fig:Topological change}).
We used time step $\delta t = 0.0039$ and spatial grid size $dx = 0.0024$.
    \begin{figure}[h!t]
    \centering{
    	\includegraphics[width=.48\linewidth]{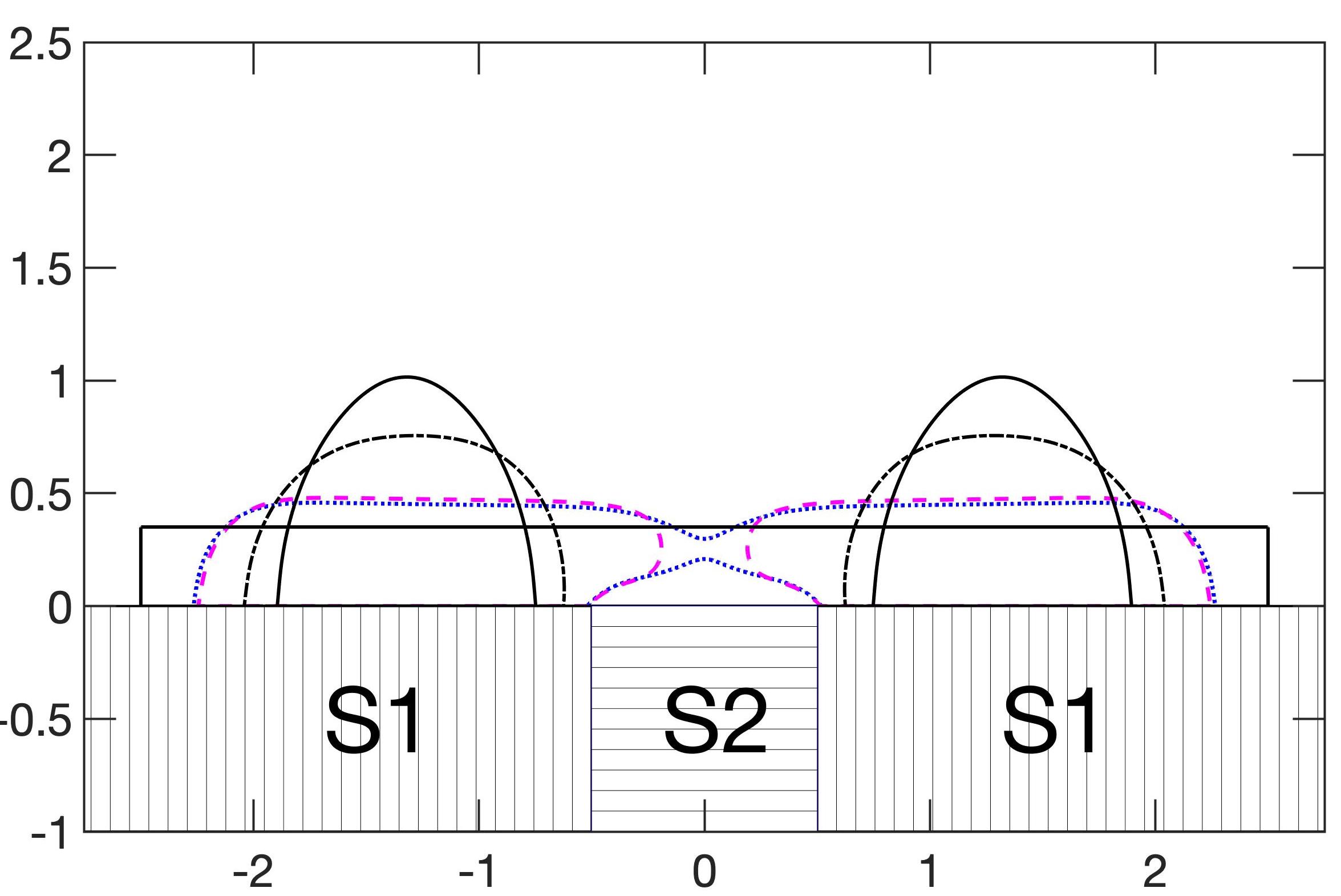}\hfill
    	\includegraphics[width=.49\linewidth]{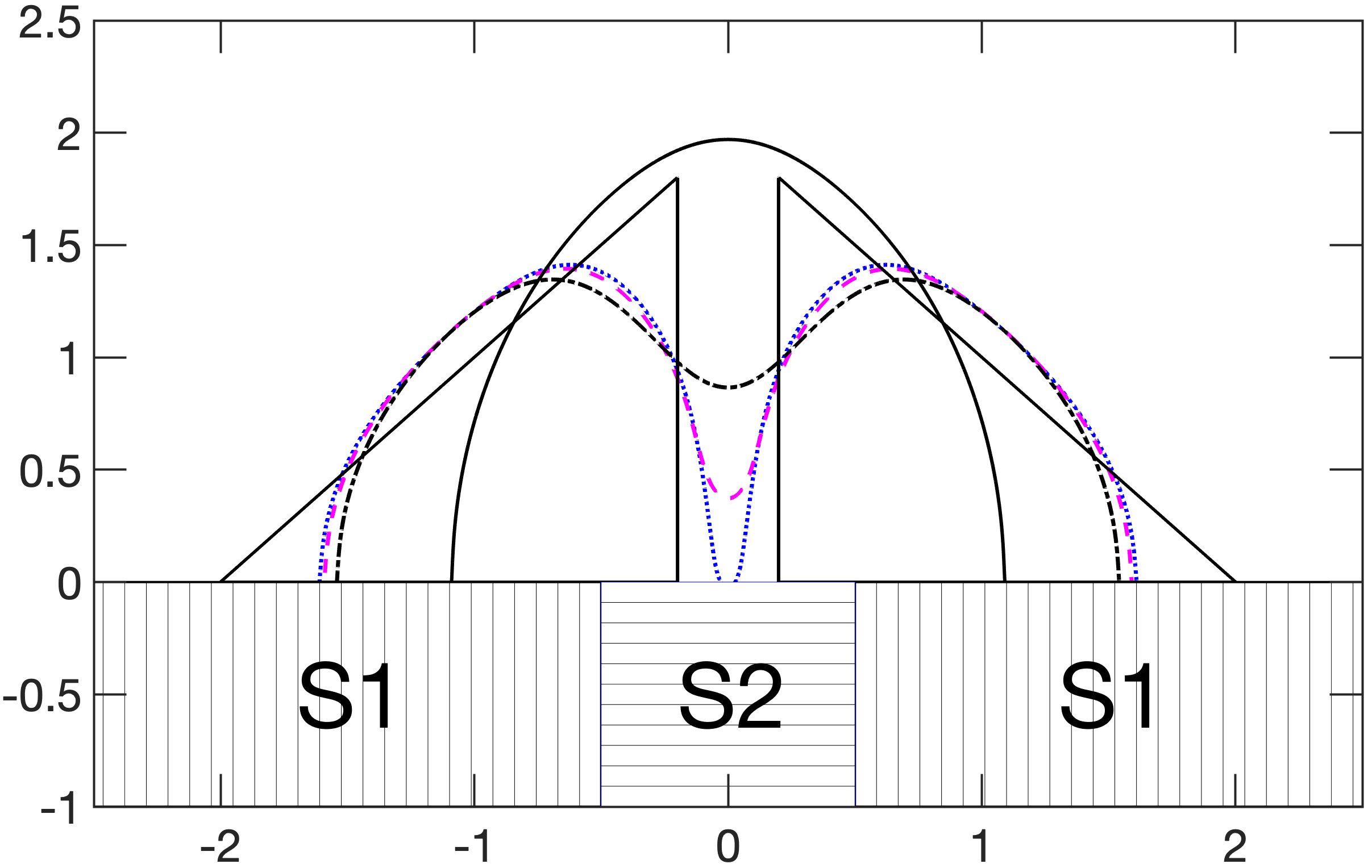}\hfill
    	\caption{Evolutions undergoing topological change for $\gamma = 1 + 0.3 \cos(2\theta + \pi)$ and $dx=0.0024$, $\delta t=  0.0039$. The initial and final shape are shown in black solid line. (Left) Splitting with $\gamma_{S_1} = 0, \gamma_{S_2} = 2$, intermediate lines showing evolution at times $25 \delta t, 28  \delta t, 70 \delta t$. (Right) Merging with $\gamma_{S_1} = 0, \gamma_{S_2} = -0.1$, intermediate lines showing evolution at times $ 49 \delta t, 55  \delta t, 75 \delta t$.}
    \label{fig:Topological change}}
    \end{figure}

As shown in Figure \ref{fig:Topological change}, in the splitting simulation particle split into two parts at 26th time step, and after that these parts were treated as two different particles with their own preserved areas.
In the merging simulation, two particles attached at 50th time step and from 51st time step on the algorithm treated them as a single particle with area equal to the sum of areas of initial particles. Note that detecting the connectivity of particles cannot be avoided if one wants to preserve the area of each particle separately.
Thanks to the symmetry of the initial configuration in our simulations, it was algorithmically easy to detect the time when topology change occurred. However, to detect topology changes happening in the evolution of a general initial configuration of particles, one needs to include a connectivity check at every time step of the algorithm.
A disadvantage of the algorithm is that due to the diffusion step, two particles sense each other even before they actually attach. This may be physically correct but the mesh size (and consequently also the time step) have to be chosen small enough to resolve or satisfactorily approximate the physically correct "sensing distance".

\section{Discussion}
We have surveyed known convolution kernels used in thresholding schemes to approximate anisotropic curvature flows and extended the scheme to realize such evolutions on obstacles.
For the basic two-phase problem, our numerical analysis confirmed the theoretically predicted first order convergence in time, and only slight differences were observed regarding the numerical performance of kernels. 
On the other hand, for the three-phase obstacle problem, although all kernels are able to approximate the correct solution to some extent, it was found that each of the kernels has certain drawbacks: 
Esedoglu-Jacobs-Zhang kernels lead to large errors, Elsey-Esedoglu kernels show spurious oscillations, while Bonnetier-Bretin-Chambolle kernel excessively smooths out sharp corners.
Our forthcoming work is going to address these issues and establish a robust scheme for multiphase anisotropic flows.\\

\noindent
{\large \bf Acknowledgement}\\
The first author acknowledges the support of the JICA's FRIENDSHIP Scholarship Program.
The research of the second author was supported by JSPS Kakenhi Grant numbers 19K03634 and 18H05481.

\vspace{2cm}
 \textbf{Name of author}: Siddharth Gavhale\\
 \textbf{E-mail address}: sb.gavhale@math.kyoto-u.ac.jp\\ 
\vspace{0.1cm}\\
\textbf{Name of author}: Karel Svadlenka\\
\textbf{E-mail address}: karel@math.kyoto-u.ac.jp\\
\vspace{0.1cm}\\
\textbf{Address of both authors}: Department of mathematics, 
            Kyoto University,\\ Kitashirakawa Oiwake-cho, Sakyo-ku,
            Kyoto, 606-8502 Japan.

\end{document}